\documentclass[10pt,leqno,oneside]{amsart}
\usepackage{a4}
\usepackage{color}
\usepackage{esint}
\usepackage{bm}

\usepackage[textsize=tiny]{todonotes}
\usepackage[backref=page]{hyperref}

\newcommand*{\mailto}[1]{\href{mailto:#1}{\nolinkurl{#1}}}

\makeatletter
\@namedef{subjclassname@2020}{\textup{2020} Mathematics Subject Classification}
\makeatother

\numberwithin{equation}{section}
\newtheorem{example}{Example}[section]
\newtheorem{definition}[example]{Definition}
\newtheorem{theorem}[example]{Theorem}
 \newtheorem{proposition}[example]{Proposition}
\newtheorem{lemma}[example]{Lemma}
 
\newtheorem{remark}[example]{Remark}
\newtheorem*{maintheorem*}{Main Theorem}
\allowdisplaybreaks
\numberwithin{equation}{section}

\renewcommand{\i}{\ifmmode\mathit{\mathchar"7010 }\else\char"10 \fi}
\renewcommand{\j}{\ifmmode\mathit{\mathchar"7011 }\else\char"11 \fi}
\newcommand{\R}{\mathbb{R}}

\DeclareMathOperator*{\sign}{sign}
\newcommand{\sgn}[1]{\sign\left(#1\right)}
\DeclareMathOperator*{\dvv}{div}
\newcommand{\dv}[1]{\dvv\left(#1\right)}

\newcommand{\abs}[1]{\left|#1\right|}
\newcommand{\absb}[1]{\bigl|#1\bigr|}
\newcommand{\norm}[1]{\left\|#1\right\|}
\newcommand{\test}{\varphi}
\newcommand{\vfi}{\varphi}
\newcommand{\seq}[1]{\left\{#1\right\}}

\newcommand{\px}{\partial_x}

\newcommand{\pt}{\partial_t}
\newcommand{\eps}{\varepsilon}

\newcommand{\ua}{u_\alpha}

\newcommand{\lipnorm}[1]{\abs{#1}_{\mathrm{Lip}}}
\newcommand{\lenloc}{L^1_{\mathrm{loc}}}
\DeclareMathOperator{\aud}{div}
\newcommand{\auda}{\aud_\alpha^\Phi\!}
\newcommand{\audaa}{\aud^{\Phi_\alpha}\!}
\newcommand{\audp}{\aud^\Phi \!}
\newcommand{\EO}{\mathrm{EO}}
\newcommand{\bn}{\bm{n}}
\newcommand{\bx}{\bm{x}}
\newcommand{\by}{\bm{y}}
\newcommand{\bz}{\bm{z}}
\newcommand{\ob}[1]{\overline{#1}}
\newcommand{\ak}{\alpha_k}
\newcommand{\uak}{u_{\ak}}
\newcommand{\loc}{\operatorname{loc}}
\DeclareMathOperator*{\essup}{ess\,sup}
\DeclareMathOperator*{\essinf}{ess\,inf}

\begin{document}

\title[Upwind filtering of conservation laws]
{Upwind filtering of scalar conservation laws}

\author[Coclite]{G. M. Coclite} \author[Karlsen]{K. H. Karlsen}
\author[Risebro]{N. H. Risebro}

\address[Giuseppe Maria Coclite] {Dipartimento di Meccanica,
  Matematica e Management, Politecnico di Bari, Via E. Orabona 4 --
  70125 Bari, Italy}
\email[]{\mailto{giuseppemaria.coclite@poliba.it}}

\address[Kenneth Hvistendahl Karlsen, Nils Henrik Risebro] {Department
  of Mathematics, University of Oslo, Postboks 1053, Blindern -- 0316
  Oslo, Norway} \email[]{\mailto{kennethk@math.uio.no},
  \mailto{nilshr@math.uio.no}}

\date{\today}



\begin{abstract}
We study a class of multi-dimensional non-local 
conservation laws of the form
$\pt u =\audp \bm{F}(u)$, where the standard local divergence
($\aud$) of the flux vector $\bm{F}(u)$ is replaced by an average
upwind divergence operator $\audp$ acting on the flux along a
continuum of directions given by a reference measure and a filter
$\Phi$.  The non-local operator $\audp$ applies to a general
non-monotone flux $\bm{F}$, and is constructed by decomposing
the flux into monotone components according to wave speeds
determined by $\bm{F}'$.  Each monotone component is then
consistently subjected to a non-local derivative operator that
utilizes an anisotropic kernel supported on the ``correct" half of
the real axis.  We establish well-posedness, derive a priori and
entropy estimates, and provide an explicit continuous dependence
result on the kernel.  This stability result is robust with respect
to the ``size" of the kernel, allowing us to specify $\Phi$ as a
Dirac delta $\delta_0$ to recover entropy solutions of the local
conservation law $\pt u =\aud \bm{F}(u)$ (with an error estimate).
Other choices of $\Phi$ (and the reference measure) recover known
numerical methods for (local) conservation laws.
This work distinguishes itself from many others in the field by 
developing a consistent non-local approach capable of 
handling non-monotone fluxes.
\end{abstract}

\maketitle

\tableofcontents

\section{Introduction}\label{sec:intro}

In recent years, the study of non-local conservation laws has gained
much attention, driven by their wide applicability across various
domains such as sedimentation processes \cite{Betancourt:2010kx},
supply chain dynamics \cite{Colombo:2011aa}, and models for pedestrian
flow \cite{Colombo:2012aa,Piccoli:2011aa} and vehicular traffic
\cite{Chiarello:0aa}, to name just a few examples.  The mathematical
analysis of conservation laws with non-local fluxes has also advanced
significantly, with notable progress in understanding existence,
uniqueness, stability, and the convergence of numerical methods.
While providing an exhaustive list of references is beyond the scope
of this work, we highlight the contributions in
\cite{BlandinGoatin:16,Coclite:2017aa,Crippa:2013aa,
  Du:2017aa,Keimer:2017aa,Keimer:2018aa,Keimer:2018ab,
  Kloeden:2016aa,Lorenz:2020aa} and refer the reader to the references
therein for additional details.  Very recently, significant attention
has also been devoted to the surprisingly tricky singular limit
problem of establishing the strong convergence of non-local
conservation laws to their local counterparts, a topic explored
extensively in \cite{Bressan:2020aa,Bressan:2021aa,
  Coclite:2023ab,Coclite:2021aa, Colombo:2021aa,Colombo:2019aa,
  Keimer:2019aa,Zumbrun:1999aa}.

In the context of traffic flow modeling, many researchers (see, e.g.,
\cite{BlandinGoatin:16,ChiarelloGoatin:18,
  FriedrichEtal:18,GoatinScialanga:16}) have investigated non-local
generalizations of the classical Lighthill-Whitham-Richards (LWR)
model \cite{Lighthill:1955aa,Richards:1956aa}.  These non-local
extensions incorporate the look-ahead distance of drivers, presumably
offering a more accurate representation of driver behavior.  Some
approaches assume that drivers react to the average downstream traffic
\textit{density}, while others model their response to the average
downstream \textit{velocity}. 
The resulting non-local LWR models take the form:
\begin{equation}\label{eq:non-localCL-intro}
  \partial_t u
  +\partial_x \left(u V\left(A_\alpha^+u\right)\right) = 0,
  \qquad 
  \partial_t u + \partial_x \left(u A_\alpha^+(V(u))\right) = 0,
\end{equation}
where $V = V(u)$ is a given decreasing speed function.  For an
integrable function $v = v(x)$, the non-local 
operator $A_\alpha^+$ is defined as
\begin{equation}\label{eq:1D-average}
	A_\alpha^+v(x) := \int_0^\infty
	\Phi_\alpha(y) v(x+y)\,dy.
\end{equation}
The kernel (filter) $\Phi_\alpha$ characterizes the non-local
influence, with its shape controlled by the ``filter size'' parameter
$\alpha > 0$.  The kernel $\Phi_\alpha$ is a non-negative,
non-increasing function defined on the non-negative real numbers, and
it is normalized to have unit mass.  This formulation effectively
captures the impact of \textit{downstream} traffic on driver behavior
while preserving mathematical consistency with the LWR model.

By setting 
\begin{equation}\label{eq:Jplus}
	J^+_\alpha(x):=
	\Phi_\alpha(-x)\chi_{(-\infty,0]}(x), 	
\end{equation}
the function $A_\alpha^+v(x)$ can be expressed as the convolution
$(J^+_\alpha \star v)(x)$.  The family of \textit{anisotropic} kernels
$\{J^+_\alpha(x)\}_{\alpha > 0}$ acts as an approximate identity on
$\R$, although it is discontinuous at $x = 0$. In the formal
limit $\alpha \to 0$ (referred to as the \textit{zero-filter} limit),
the non-local fluxes $u V\left(A_\alpha^+{u}\right)$ and
$u A_\alpha^+{V(u)}$ reduce to the local 
flux $u V(u)$, recovering the original LWR equation

In \cite{Coclite:2023aa}, we introduced a non-local
Follow-the-Leader model for traffic flow in \emph{Lagrangian coordinates},
where the velocity of each vehicle is influenced by the surrounding
downstream traffic density.  This approach employs a \emph{weighted harmonic
mean} to model the non-local interactions, resulting in a non-local
Lagrangian PDE for the spacing between vehicles, represented as
$y = 1/u$ (the road length per car).  
The macroscopic equation takes the non-local form
\begin{equation}\label{eq:y-nofilter}
	\partial_t y - \partial_z W\left(A_\alpha^+y\right) = 0, 
	\quad
	 W(y) := V\left(\frac{1}{y}\right),
\end{equation}
where $\Phi_\alpha(\zeta) =\frac{1}{\alpha}\Phi(\frac{\zeta}{\alpha})$
is a filter characterizing the non-locality and ensuring that the
weighted harmonic mean is applied consistently; indeed, as the filter
size $\alpha$ tends to zero, the model recovers the local Lagrangian
PDE $\partial_t (1/u) - \partial_z V(u) = 0$, the Lagrangian
formulation of the LWR model \cite[Section 5]{Coclite:2023aa}.

Importantly, the filtered variable $w := A_\alpha^+{y}$ satisfies the
non-local conservation law
\begin{equation}\label{eq:1D-filter-prev}
  \partial_t w = \partial_z A_\alpha^+({W(w)}),
\end{equation}
which incorporates the weighted harmonic mean and differs from more
conventional non-local models like \eqref{eq:non-localCL-intro}.  This
model and its analysis in \cite{Coclite:2023aa}, like most
traffic-inspired non-local (scalar) equations, fundamentally rely on
the monotonicity of the flux function $W(\cdot)$, which is a realistic
assumption in traffic flow modeling as the speed function $V(\cdot)$
is typically non-increasing.

In this paper, we aim to broaden the modeling framework based on
\eqref{eq:1D-average} by developing a non-local conservation law that
accommodates general non-monotone flux functions.  Furthermore, we
extend the model to multiple spatial dimensions, expanding its
applicability to areas beyond traffic flow.
The innovative aspect of our work lies in its consistent 
non-local framework, which effectively resolves the challenges 
posed by non-monotone fluxes, a limitation in current 
non-local models inspired by traffic flow (see the 
previously mentioned references).

For a scalar function $f$, we use $f^\pm$ 
to denote respectively the increasing and decreasing parts of $f$, 
as defined in \eqref{eq:EO-split}. Consider a flux vector $\bm{F}(u)$. 
The functions
$$
\frak{f}^\pm(\by,u)
=\pm\left[\bm{F}\cdot \by\right]^\pm(u),
$$
represent the increasing ($+$) and decreasing ($-$) components of the
flux $\bm{F}$ along the direction $\by \in \R^d$.  The non-local
conservation law model proposed in this paper takes the form
\begin{equation}\label{eq:non-local-CL}
	\pt u =\audp \bm{F}(u), 
	\quad x\in \R^d, \,\, t>0,
\end{equation}
where, given a filter $\Phi$, the 
so-called \emph{average upwind divergence} 
operator $\audp$ is defined by
\begin{equation}\label{eq:non-local-CL2}
  \begin{split}
    \audp \bm{F}(u)(\bx,t) & =
    \int_{\R^d}\frac{\Phi'(\abs{\by})}{\abs{\by}^d}
    \Bigl\{\frak{f}^+(\by,u(\bx,t))-\frak{f}^+(\by,u(\bx+\by,t)) 
    \\ &
    \qquad\qquad\qquad\quad +\frak{f}^-(\by,u(\bx,t))
    -\frak{f}^-(\by,u(\bx-\by,t))\Bigr\} \,d\bm{w}(\by),
  \end{split}
\end{equation}
for $\bx\in\R^d$ and $t>0$.  Here, $\bm{w}$ 
is a Radon measure that, 
in radial coordinates, is expressed as
$d\bm{w}(\by) = r^{d-1} drdw(\theta)$ with $\by=r
\bm{n}_\theta$. The measure $w$, defined on the unit sphere $S^{d-1}$,
is carefully normalized so that the weighted average of directional
contributions corresponds to the identity matrix.  This normalization
ensures that specific directions can be emphasized without introducing
bias to the overall contributions across all spatial dimensions (see
\eqref{eq:scaling}), making the model consistent with the local
conservation law
\begin{equation}\label{eq:local-CL}
  \pt u =\aud\bm{F}(u),
  \quad x\in \R^d, \,\, t>0.
\end{equation}

The non-local conservation law defined by \eqref{eq:non-local-CL} and
\eqref{eq:non-local-CL2} models the evolution of a scalar quantity
$u(\cdot,t)$, incorporating spatial interactions through the average
upwind divergence operator $\audp$.  This operator uses a kernel
$\Phi$, a non-negative, non-increasing function on $[0,\infty)$ with
unit mass and finite first moment, to define the filtering effect, and
a measure $\bm{w}$ that emphasizes specific spatial directions.

When considering new non-local models, ensuring the consistency with
their local counterparts is essential but nontrivial: Does replacing
the filter with a Dirac delta lead to the corresponding local
conservation law?  For equations of the form (from the first part of
\eqref{eq:non-localCL-intro})
\begin{equation}\label{eq:standard-model}
  \partial_t u_\alpha + \partial_x 
  \left(u_\alpha V(u_\alpha \star J_\alpha)\right) = 0,
\end{equation}
where the velocity part of the flux depends on the convolution of
$u_\alpha$ with $J_\alpha$, the formal singular limit replacing
$J_\alpha$ with a Dirac delta $\delta_0$ leads to the local equation
\eqref{eq:local-CL} with flux $\bm{F}(u) = uV(u)$ (and
$d=1$). However, counterexamples \cite{Colombo:2019aa} show that
strong convergence of solutions is not guaranteed for general 
$V$. Recent results \cite{Coclite:2021aa,Colombo:2021aa} 
demonstrate that if $V(\cdot)$ is monotone non-increasing and the (convex)
convolution kernel $J_\alpha$ is supported on $(-\infty, 0]$ (as in
traffic flow modeling), the total variation of $u_\alpha(\cdot, t)$ is
uniformly bounded with respect to $\alpha$.  Furthermore, the
solutions $u_\alpha$ converge in $L^1_{\loc}$ to the unique entropy
solution of the local conservation law as $\alpha \to 0$, with entropy
solutions understood in the sense of \cite{Kruzkov:1970kx}.

Similar to \eqref{eq:standard-model}, convolution-based conservation
laws of the form \eqref{eq:1D-filter-prev} also face consistency
issues. However, in \cite{Coclite:2023aa}, we rigorously verified the
zero-filter limit for \eqref{eq:1D-filter-prev}, provided the flux
$W(\cdot)$ is monotone (non-increasing).

For our multidimensional model \eqref{eq:non-local-CL} and
\eqref{eq:non-local-CL2} with a general non-monotone flux $\bm{F}$,
consistency with \eqref{eq:local-CL} (i.e., the zero-filter limit)
will be established due to the upwind nature of the non-local
divergence operator $\audp$ and an anisotropic filter $\Phi$ supported
on the ``correct" half of $\R$. The flux components $\frak{f}^\pm$
capture increasing and decreasing contributions to the flux vector
$\bm{F}$, respectively, with interactions at a spatial point $\bx$
adjusted by upwind shifts $\pm \by$ based on monotonicity.  The
operator $\audp$ aggregates non-local flux differences over all
directions and distances, weighted by the kernel $\Phi$ and the
measure $\bm{w}$.  Overall, the new equation \eqref{eq:non-local-CL} is
designed to capture anisotropic and directional non-local effects in
general multi-dimensional transport-type dynamics.

Equation \eqref{eq:non-local-CL} can be interpreted as an
infinite-dimensional system of ODEs, from which the existence and
uniqueness of (Lipschitz) solutions follow naturally. We will
demonstrate that the structure of the non-local divergence operator
$\audp$ ensures that solutions of \eqref{eq:non-local-CL} dissipate all
convex entropies by deriving a precise entropy balance equation. These
non-local entropy inequalities are then employed to establish a key
$L^1$ stability or continuous dependence estimate with respect to
perturbations in the initial data and the filter.  Using the stability
estimate, we derive a series of uniform $L^1$, $L^\infty$, and $BV$
bounds for the solutions that are independent of the ``size" of the
filter $\Phi$.

Let $u^{(1)}$ and $u^{(2)}$ be solutions to \eqref{eq:non-local-CL}
corresponding to initial data $u_0^{(1)}$ and $u_0^{(2)}$, and filters
$\Phi^{(1)}$ and $\Phi^{(2)}$, respectively.  The continuous
dependence estimate reads
\begin{equation*}
  \begin{aligned}
    \norm{u^{(1)}(\cdot,t)-u^{(2)}(\cdot,t)}_{L^1(\R^d)}&\le
    \norm{u_0^{(1)}-u_0^{(2)}}_{L^1(\R^d)}\\
    &\qquad + C\sqrt{\int_0^\infty
      \abs{r\Phi^{(1)}(r)-r\Phi^{(2)}(r)}\,dr}.
  \end{aligned}
\end{equation*}  
Here, the constant $C$ depends on $t$ and the $BV$ seminorm of the 
initial data. This estimate represents the central
technical contribution of the paper and is robust in its explicit
dependence on the filter.

Specifically, it allows us to set $u_0^{(1)} = u_0^{(2)}$ and take
$\Phi^{(2)}$ as the Dirac delta $\delta_0$ (interpreting the product
$r\Phi^{(2)}(r)$ as the zero function), while choosing $\Phi^{(1)}$ as
the rescaled filter
$\Phi_\alpha(r):=\frac{1}{\alpha} \Phi\left(\frac{r}{\alpha}\right)$,
which approximates $\delta_0$.  In this setting, we recover the unique
entropy solution $u$ of the local conservation law
\eqref{eq:local-CL}, along with the error estimate
$$
\norm{u_\alpha(\cdot,t)-u(\cdot,t)}_{L^1(\R^d)} \leq C\sqrt{\alpha},
$$
where $u_\alpha$ denotes the solution of the non-local equation
\eqref{eq:non-local-CL} with $\Phi$ from \eqref{eq:non-local-CL2}
replaced by the rescaled filter $\Phi_\alpha$.

In the one-dimensional case ($d=1$), with $\bm{F} = f$, the average
upwind divergence $\audaa$ simplifies 
to (see Lemma \ref{lm:d=1op}):
\begin{equation}\label{eq:op11}
	\begin{split}
		\audaa{f(v)}(x) 
		& = \int_0^\infty \Phi_\alpha'(y)
		\bigl(f^+(v(x + y)) - f^+(v(x))\bigr) \, dy 
		\\ & \qquad 
		+ \int_{-\infty}^0 \Phi_\alpha'(-y) 
		\bigl(f^-(v(x+y)) - f^-(v(x))\bigr) \, dy,
	\end{split}
\end{equation}
for a function $v \in L^1(\R)$, where $\Phi_\alpha$ is, as before, the
rescaled version of the filter $\Phi$.  Formally, by applying
integration by parts, \eqref{eq:op11} can be 
rewritten in the suggestive form:
\begin{equation}\label{eq:suggestive}
	\audaa{f(u)}(x)
	= \partial_x A^+_\alpha(f^+(v))
	+\partial_xA^-_\alpha (f^-(v)),	
\end{equation}
where $A_\alpha^+$ is defined 
in \eqref{eq:1D-average} and
\begin{equation}\label{eq:A-}
	A_\alpha^-{v}(x) 
	= \int_{-\infty}^0 \Phi_{\alpha}(-\zeta)
	v(x+\zeta)\,d\zeta.
\end{equation}
Defining (compare with \eqref{eq:Jplus})
\begin{equation}\label{eq:Jmin}
	J^-_\alpha(x):=
	\Phi_\alpha(x)\chi_{(0,\infty)}(x),
\end{equation}
we obtain that $A_\alpha^-{v}(x) =(J^-_\alpha \star v)(x)$. 

The formulation \eqref{eq:suggestive} 
highlights that the resulting equation
\eqref{eq:non-local-CL} naturally extends the non-local equation
\eqref{eq:1D-filter-prev} introduced in \cite{Coclite:2023aa} (as a
traffic flow model) to the general case of non-monotone fluxes $f$. 
In the one-dimensional scenario involving a non-monotonic 
$f$, our non-local approach employs two upwind-adapted 
discontinuous kernels, $J^+_\alpha$ and $J^-_\alpha$, as specified 
in \eqref{eq:Jplus} and \eqref{eq:Jmin}. These kernels are used 
to approximate the nonlinear transport 
operator $\partial_x f(v)$ by
$$
\partial_x f(v)\approx 
\partial_x \left(J_\alpha^+\star f^+(v)\right)
+\partial_x \left(J_\alpha^-\star f^-(v)\right).
$$

If we choose $\Phi$ as the characteristic function of the interval
$[0,1]$ in \eqref{eq:suggestive}, the non-local equation
\eqref{eq:non-local-CL} transforms into:
\begin{equation*}
  \pt \ua(x,t)=
  \frac{1}{\alpha}\left(f^\EO(\ua(x,t),\ua(x+\alpha,t))
    -f^\EO(\ua(x-\alpha,t),\ua(x,t))\right),
\end{equation*}
which represents a semi-discrete formulation of the well-known
Engquist-Osher scheme for (local) conservation laws
\cite{EngquistOsher:81}.  By choosing different forms of $\Phi$, one
can derive alternative numerical schemes.  Appropriate choices of the
filter $\Phi$ and the measure $\bm{w}$ in \eqref{eq:non-local-CL},
\eqref{eq:non-local-CL2} can produce finite volume schemes in the
multi-dimensional case.

\medskip

The remainder of this paper is organized as follows: In Section
\ref{sec:rotat}, we introduce the necessary notation and provide a
precise definition of the average upwind divergence operator. Section
\ref{sec:main} focuses on the well-posedness of the non-local model,
deriving a priori estimates and establishing the continuous dependence
estimate. In Section \ref{sec:filtlim}, we examine the zero-filter
limit.  Section \ref{sec:examp} discusses various examples of
equations and numerical methods derived from the non-local equation.
Finally, Section \ref{sec:bonus} contains material that 
offers additional perspectives on the approach.

\section{Average upwind divergence}
\label{sec:rotat}

We begin by defining the notation to be used and the averaged upwind
divergence operator.  For a continuous function $f:\R\to\R$, we denote
the increasing and decreasing parts of $f$ as $f^+$ and $f^-$,
respectively; more precisely,
\begin{equation}\label{eq:EO-split}
  f^+(u)=\int_0^u \max\seq{0,f'(s)}\,ds,
  \qquad  
  f^-(u)=\int_0^u \min\seq{0,f'(s)}\,ds.
\end{equation}
Clearly, we have $f=f^++f^-$.

In what follows, let $\bm{F}:\R\to\R^d$ be a locally Lipschitz
continuous (flux) function satisfying $\bm{F}(0)=\bm{0}$, where
$\bm{F}$ has components
\begin{equation*}
  \bm{F}=\begin{pmatrix}f_1\\ 
    \vdots\\ f_d\end{pmatrix}.
\end{equation*}

Let the functions $\frak{f}^\pm$ be defined as
\begin{equation}\label{eq:frakfdef}
  \frak{f}^\pm(\by,u)=\pm\left[\bm{F}
    \cdot \by\right]^\pm(u),
\end{equation}
for $\by\in\R^d$ and $u\in\R$.  In the following, $S^{d-1}$ denotes
the $(d-1)$-dimensional boundary of the unit ball in $\R^d$.  For a
point $\theta\in S^{d-1}$, $\bm{n}_\theta$ denotes the outward
pointing unit normal at $\theta$, so that for any $\by\in\R^d$, we can
write $\by=\abs{\by}\bm{n}_\theta$. We shall frequently use the
identity
\begin{equation*}
  \frak{f}^\pm(\by,u)=\abs{\by}
  \frak{f}^\pm(\bm{n}_\theta,u).
\end{equation*}

Throughout the paper, we will assume that the filter $\Phi$ is a
non-negative and non-increasing function defined on the positive real
numbers $\R^+$.  At various occasions, we will need some or all of the
following additional conditions:
\begin{align}
  \label{eq:Phi-r}
  \int_0^\infty \Phi(\xi)\,d\xi&=1,\\
  \label{eq:Phi-m}
  \int_0^\infty \Phi(\xi)\xi\,d\xi&<\infty,\\
  \Phi(0)&<\infty,\label{eq:Phi-zero}\\
  \intertext{as well as the weaker condition}
  \label{eq:Phi-zerolim}
  \lim_{\eps\to 0} \Phi(\eps)\eps &=0.
\end{align}
If $\Phi$ has a jump discontinuity at some $r_0>0$, such that the
limits $\Phi_{\pm}=\lim_{r\to r_0^\pm}\Phi(r)$ exist, then
$\Phi'(r_0)$ is interpreted as a Dirac measure located at $r_0$ with
mass $\Phi_+-\Phi_-$.

We shall be using the following lemma repeatedly
\begin{lemma}\label{lem:Phibp}
  Suppose $\Phi$ is a non-negative and non-increasing function on
  $\R^+$ that satisfies \eqref{eq:Phi-r} and \eqref{eq:Phi-zerolim}.
  Let $u=u(x)$ be a Lipschitz continuous function on $\R$. Then
  \begin{equation*}
    \frac{d}{dx} \int_0^\infty
    \Phi(r)u(x+r)\,dr =
    \int_0^\infty \Phi'(r)
    \left(u(x)-u(x+r)\right)\,dr.
  \end{equation*}
\end{lemma}

\begin{proof}
  We have
  \begin{equation*}
    \frac{d}{dr}\left(\Phi(r)
      \left(u(x)-u(x+r)\right)\right)
    =\Phi'(r)\left(u(x)-u(x+r)\right)-\Phi(r)u'(x+r).
  \end{equation*}
  Integrate this from $r=\eps$ to $r=\infty$ to get
  \begin{align*}
    -\Phi(\eps)\left(u(x)-u(x+\eps)\right)
    &=\int_\eps^\infty \Phi'(r)\left(u(x)-u(x+r)\right)\,dr
    \\ &\qquad -\int_\eps^\infty\Phi(r)u'(x+r)\,dr.
  \end{align*}
  The proof is concluded by observing that since $u$ is Lipschitz
  continuous and \eqref {eq:Phi-zerolim} holds, the limit of the
  left-hand side as $\eps\to 0$ is zero.
\end{proof}

Let $w$ be a non-negative Radon measure on $S^{d-1}$, scaled such that
\begin{equation}\label{eq:scaling}
  \int_{S^{d-1}} \left(\bm{n}_\theta\cdot 
    \bm{e}_i\right)\left(\bm{n}_\theta
    \cdot \bm{e}_j\right)\,dw(\theta)
    =\delta_{ij},
    \quad i,j =1,\ldots,d,
\end{equation}
where $\bm{e}_i$ denotes the $i$th unit basis vector in $\R^d$. The
measure $w$ will be used to emphasise certain spatial directions. Let
$\bm{w}$ denote the corresponding measure defined in radial
coordinates as follows:
\begin{equation*}
  \by=r \bm{n}_\theta,\qquad 
  d\bm{w}(\by)=r^{d-1}drdw(\theta). 
\end{equation*}

We call the mapping
\begin{equation}\label{eq:non-local-div}
  \begin{split}
    \bm{F}(u(\bx)) &\mapsto \int_{\R^d}
    \frac{\Phi'(\abs{\by})}{\abs{\by}^d} \Bigl\{\frak{f}^+(\by,u(\bx))
    -\frak{f}^+(\by,u(\bx+\by)) \\ &\hphantom{\int_{\R^d}
      \frac{\Phi'(\abs{\by})}{\abs{\by}^d} \Bigl\{}\qquad\quad
    +\frak{f}^-(\by,u(\bx)) -\frak{f}^-(\by,u(\bx-\by))\Bigr\}
    \,d\bm{w}(\by)\\
    &\qquad =: \audp \bm{F}(u)(\bx),
  \end{split}
\end{equation}
the \emph{average upwind divergence} of $\bm{F}(u(\bx))$.  Hence, the
term ``upwind filtering''.  The map $u\mapsto \audp \bm{F}(u)$ is well
defined for $u\in L^1(\R^d)$.

We conclude this section by demonstrating that 
in the one-dimensional case, the operator described 
in \eqref{eq:non-local-div} simplifies to the one in 
\eqref{eq:suggestive}. The one-dimensional case is 
distinctive because $S^0=\seq{1,-1}$, and 
the (counting) measure $w$ is 
characterized by $w(\pm1)=\frac{1}{2}$.

\begin{lemma}\label{lm:d=1op}
Suppose $\Phi$ is a non-negative and non-increasing function on
$\R^+$ that satisfies \eqref{eq:Phi-r} and \eqref{eq:Phi-zerolim}.
Let $u=u(x)$ be a Lipschitz continuous function on $\R$. 
For $d=1$, the operator \eqref{eq:non-local-div} simplifies to
\begin{equation}\label{eq:1D-operators}
	\begin{split}
		\audp\bm{F}(u)(x)
		&=\partial_x\int_0^\infty\Phi(r)f^+(u(x+r))\,dr
		\\ & \qquad 
		+\partial_x\int_{-\infty}^0\Phi(-r)f^-(u(x+r))\,dr
		\\ & =\partial_x \int_0^\infty
		\Phi(r)\Bigl(f^+(u(x+r))+f^-(u(x-r))\Bigr)\,dr.
	  \end{split}	
\end{equation}
\end{lemma}

\begin{proof}
To verify \eqref{eq:1D-operators}, 
we proceed as follows:
\begin{align*}
	& \audp\bm{F}(u)(x)
	\\ & \quad 
	= \int_{\{n_\theta=\pm 1\}}\int_0^\infty
	\Phi'(r) \Bigl((fn_\theta)^+(u(x))
	-(fn_\theta)^+(u(x+rn_\theta))
	\\ &\qquad\qquad\qquad\qquad\quad
	-(fn_\theta)^-(u(x))+(fn_\theta)^-(u(x-rn_\theta))\Bigr)
	\,drd\theta
	\\ & \quad
	= \frac{1}{2}\int_0^\infty
	\Phi'(r) \Bigl(f^+(u(x))-f^+(u(x+r))
	\\ &\qquad\qquad\qquad\qquad\quad
	-f^-(u(x)) +f^-(u(x-r))\Bigr)\,dr
	\\ & \quad\qquad 
	+\frac{1}{2}\int_0^\infty
	\Phi'(r) \Bigl(-f^-(u(x))+f^-(u(x-r))
	\\ &\qquad\qquad\qquad \qquad\qquad
	+f^+(u(x)) -f^+(u(x+r))\Bigr)\,dr
	\\ & \quad
	= \int_0^\infty
	\Phi'(r) \Bigl(f^+(u(x))-f^+(u(x+r))\Bigr)\,dr
	\\ & \quad \qquad 
	-\int_0^\infty\Phi'(r)\Bigl(f^-(u(x))-f^-(u(x-r))\Bigr)\,dr
	\\ & \quad 
	= \int_0^\infty \Phi'(r) \Bigl(f^+(u(x))-f^+(u(x+r))\Bigr)\,dr
	\\ & \quad\qquad
	-\int_{-\infty}^0 \Phi'(-r) \Bigl(f^-(u(x))-f^-(u(x+r))\Bigr)\,dr.
\end{align*}
Integrating by parts, we obtain
\begin{align*}
	\audp\bm{F}(u)(x)=\int_0^\infty
	\Phi(r)\px f^+(u(x+r))\,dr
	+\int_{-\infty}^0
	\Phi(-r)\px f^-(u(x+r))\,dr,
\end{align*}
which yields the first claim in \eqref{eq:1D-operators}.

By changing the variable in the second integral from $r$ to $-r$, 
both integrals are now on the interval $(0,\infty)$, 
which proves the second claim.
\end{proof}

\section{Upwind filtered conservation laws}
\label{sec:main}

Given the average upwind divergence operator \eqref{eq:non-local-div},
the Cauchy problem we shall study has the form
\begin{equation}\label{eq:rotfilter2}
  \begin{split}
    \partial_t u(\bx,t) & =\audp \bm{F}(u)(\bx) \\ & =\int_{\R^d}
    \frac{\Phi'(\abs{\by})}{\abs{\by}^d}
    \Bigl\{\frak{f}^+(\by,u(\bx,t)) -\frak{f}^+(\by,u(\bx+\by,t)) \\ &
    \qquad\qquad\qquad\qquad +\frak{f}^-(\by,u(\bx,t))
    -\frak{f}^-(\by,u(\bx-\by,t))\Bigr\} \,d\bm{w}(\by),
    \\
    u(\bx,0) & =u_0(\bx),
  \end{split}
\end{equation}
for $\bx\in\R^d$ and $t>0$, where the functions $\frak{f}^\pm$ 
are defined in \eqref{eq:frakfdef}.

In this section, we will first verify the well-posedness 
of the upwind filtered 
conservation law \eqref{eq:rotfilter2}.

\begin{proposition}\label{prop:L1exist}
  Suppose the filter $\Phi$ is a non-negative, non-increasing function
  on $\R^+$ that meets the condition \eqref{eq:Phi-zero}.  Assume that
  the function $u\mapsto \bm{F}(u)$ is (globally) Lipschitz continuous
  with Lipschitz constant $L$. If $u_0\in L^\infty(\R^d)$, then there
  exists a unique continuously differentiable solution
  $u(t)\in L^\infty(\R^d)$ to the Cauchy problem \eqref{eq:rotfilter2}
  on $[0,\infty)$.

  Furthermore, if $\lipnorm{u_0}<\infty$, then $u(t)$ is Lipschitz
  continuous in $\bx\in \R^d$, with the explicit bound
  \begin{equation*}
    \lipnorm{u(\cdot,t)}\le \lipnorm{u_0}
    \exp\left(4Lw(S^{d-1})
      \Phi(0)t\right).
  \end{equation*}
\end{proposition}

\begin{proof}
  We view the map $u\mapsto \audp \bm{F}(u)$ as a map from
  $L^{\infty{}}(\R^d)$ to itself.  Thus \eqref{eq:rotfilter2} has the
  form of an ordinary differential equation in $L^{\infty}(\R^d)$, and
  using \cite[Theorem~16.2]{Pata:2019aa}, existence of a unique
  solution follows if the right hand side is Lipschitz continuous in
  $L^{\infty}(\R^d)$. We use the notation
  \begin{equation*}
    \Delta(\bx)=u(\bx)-v(\bx)\ \ \text{and}\ \
    \Delta^\Phi(\bx)=\audp 
    \bm{F}(u)(\bx)-\audp\bm{F}(v)(\bx).
  \end{equation*}
  Then we have
  \begin{align*}
    \abs{\Delta^\Phi(\bx)}
    &=\Bigl|\int_{\R^d} 
      \frac{\Phi'(\abs{\by})}{\abs{\by}^d}
      \Bigl\{\left(\frak{f}^+(\by,u(\bx))-
      \frak{f}^+(\by,v(\bx))\right)
    \\
    &\qquad\qquad
      -\left(\frak{f}^+(\by,u(\bx+\by))
      -\frak{f}^+(\by,v(\bx+\by))\right)\\
    &\qquad\qquad
      +\left(\frak{f}^-(\by,u(\bx))
      -\frak{f}^-(\by,v(\bx))\right)\\
    &\qquad\qquad -\left(\frak{f}^-(\by,u(\bx-\by))
      -\frak{f}^-(\by,v(\bx-\by,t))\right)\Bigr\}
      \,d\bm{w}(\by)\Bigr|\\
    &\le -\int_{\R^d} \frac{\Phi'(\abs{\by})}{\abs{\by}^d}
      \Bigl\{\abs{\frak{f}^+(\by,u(\bx))-
      \frak{f}^+(\by,v(\bx))}\notag
    \\
    &\qquad\qquad
      + \abs{\frak{f}^+(\by,u(\bx+\by))
      -\frak{f}^+(\by,v(\bx+\by))}\notag\\
    &\qquad\qquad
      +\abs{\frak{f}^-(\by,u(\bx))
      -\frak{f}^-(\by,v(\bx))}\\
    &\qquad\qquad +\abs{\frak{f}^-(\by,u(\bx-\by))
      -\frak{f}^-(\by,v(\bx-\by))}\Bigr\}
      \,d\bm{w}(\by),\notag
  \end{align*}
  since
  \begin{equation*}
    \abs{\frak{f}^\pm(\by,u)-\frak{f}^\pm(\by,v)}
    \le L\abs{\by}\abs{u-v},
  \end{equation*}
  we get
  \begin{align*}
    \norm{\Delta^\Phi}_{L^\infty(\R^d)}
    &\le 4L \norm{\Delta}_{L^\infty(\R^d)}
      \int_{\R^d}
      \frac{-\Phi'(\abs{\by})}{\abs{\by}^{d-1}}
      \,d\bm{w}(\by)\\ &
      =4Lw(S^{d-1})\Phi(0)
      \norm{\Delta}_{L^\infty(\R^d)},
  \end{align*}
  confirming the Lipschitz continuity of $\audp\bm{F}$ in $L^\infty$.

  To show that the solution is also Lipschitz continuous in $\bx$, we
  choose $\bm{h}\ne\bm{0}$, and use $v(\bx,t)=u(\bx+\bm{h},t)$ and
  that $u$ is a solution to \eqref{eq:rotfilter2} to deduce
  \begin{align*}
    \pt
    \frac{\abs{u(\bx+\bm{h},t)-u(\bx,t)}}{\abs{\bm{h}}}
    &\le \frac{1}{\abs{\bm{h}}} \abs{\Delta^\Phi(\bx,t)}\\
    &\le 4Lw(S^{d-1})\Phi(0)\lipnorm{u(\cdot,t)}.
  \end{align*}
  Hence, by integration in $t$
  \begin{equation*}
    \frac{\abs{u(\bx+\bm{h},t)-u(\bx,t)}}{\abs{\bm{h}}}
    \le
    \lipnorm{u_0}+4Lw(S^{d-1})\Phi(0)
    \int_0^t \lipnorm{u(\cdot,s)}\,ds,
  \end{equation*}
  and thus
  \begin{equation*}
    \lipnorm{u(\cdot,t)}\le
    \lipnorm{u_0}+4Lw(S^{d-1})\Phi(0)
    \int_0^t \lipnorm{u(\cdot,s)}\,ds.
  \end{equation*}
  An application of Gronwall's inequality finishes the proof.
\end{proof}

\begin{remark}\normalfont
  It may seem restrictive to assume global Lipschitz continuity of
  $\bm{F}$. However, we shall show that solutions of
  \eqref{eq:rotfilter2} are bounded in $L^\infty$ by their initial
  data. So if we consider initial data in $L^\infty(\R^d)$, it is
  sufficient to assume that $\bm{F}$ is locally Lipschitz continuous.
\end{remark}

\begin{remark}\label{rem:L1-L1loc}
  \normalfont Observe that if $u_0\in L^1(\R^d)$, then the above proof
  shows that
  \begin{equation*}
    \pt \norm{\Delta(\cdot,t)}_{L^1(\R^d)}
    \le 4Lw(S^{d-1})
    \Phi(0)\norm{\Delta(\cdot,t)}_{L^1(\R^d)}.
  \end{equation*}
  By Gronwall's inequality, this implies that if $u_0\in L^1(\R^d)$,
  then $u(\cdot,t)\in L^1(\R^d)$ for $t>0$.  An analogous claim holds
  if $u_0\in\lenloc(\R^d)$.
\end{remark}

We next present a collection of properties, or a priori estimates, of
solutions to the non-local model \eqref{eq:rotfilter2}.  Moreover, we
establish a continuous dependence result \eqref{eq:Phistab} for
perturbations of the initial data $u_0$ and the filter $\Phi$, which
is the most novel (technical) contribution of this work due to its
robustness in the filter size, as explained in the introduction.
Since the primary focus of this study is on the non-locality of the
model \eqref{eq:rotfilter2}, we have not included stability with
respect to the flux function, which is well-understood in the context
of local conservation laws \cite{Bouchut:1998ys,Lucier:1986px}, and
can also be addressed for the non-local model in this paper.

\begin{theorem}\label{thm:properties}
  Suppose the filter $\Phi$ is a non-negative and non-increasing
  function on $\R^+$ that satisfies \eqref{eq:Phi-r} and
  \eqref{eq:Phi-zerolim}.  Let $u$ be the unique solution to the
  non-local Cauchy problem \eqref{eq:rotfilter2}, as guaranteed by
  Proposition \ref{prop:L1exist}.  This solution $u(t)$ belongs to
  $L^1(\R^d)$ or $\lenloc(\R^d)$, depending on the initial data $u_0$
  (see Remark \ref{rem:L1-L1loc}).
  \begin{itemize}
  \item[\bf{a}] The solution $u$ satisfies
    \begin{equation}
      \essinf_{\bx\in\R^d}u_0(\bx)\le 
      u(\bx,t)\le \essup_{\bx\in\R^d}u_0(\bx)\ \ \ 
      \text{for almost all $\bx$,}
      \label{eq:supbnd}
    \end{equation}
    where $\essinf$ and $\essup$ may be $\mp\infty$.  If $u_0$ is in
    $BV(\R^d)$, then, for $t>0$, $u(\cdot,t)\in BV(\R^d)$ with
    \begin{align}\label{eq:BVbnd}
      \abs{u(\cdot,t)}_{BV(\R^d)}
      &\le \abs{u_{0}}_{BV(\R^d)},\\
      \intertext{and}
      \label{eq:TLipbnd}
      \norm{\pt u(\cdot,t)}_{L^1(\R^d)} 
      & \le 2w(S^{d-1})L \abs{u_{0}}_{BV(\R^d)}.
    \end{align}

  \item[\bf{b}] If $v$ is another solution of \eqref{eq:rotfilter2}
    with initial data $v_0$ and $(u_0-v_0)\in L^1(\R^d)$, then
    \begin{equation}\label{eq:L1bnd}
      \norm{u(\cdot,t)-v(\cdot,t)}_{L^1(\R^d)}
      \le \norm{u_0-v_0}_{L^1(\R^d)},\quad t\ge 0.
    \end{equation}

  \item[\bf{c}] If $v$ is the unique solution of
    \begin{equation*}
      v_t=\aud^\Psi\!\bm{F}(v),\ \ \ v(\cdot,0)=v_0,
    \end{equation*}
    where $(u_0-v_0)\in L^1(\R^d)$, $u_0$ \emph{or} $v_0$ is in
    $BV(\R^d)$, and $\Psi$ is a non-negative and non-increasing filter
    function on $\R^+$ satisfying \eqref{eq:Phi-r} and
    \eqref{eq:Phi-zerolim}, then
    \begin{equation}\label{eq:Phistab}
      \begin{aligned}
        \norm{u(\cdot,T)-v(\cdot,T)}_{L^1(\R^d)}&\le
        \norm{u_0-v_0}_{L^1(\R^d)}\\
        &\qquad + C\Bigl(T\int_0^\infty
        r\abs{\Phi(r)-\Psi(r)}\,dr\Bigr)^{\frac12},
      \end{aligned}
    \end{equation}
    where the constant $C$ only depends on $d$,
    $\min\seq{\abs{u_0}_{BV},\abs{v_0}_{BV}}$ and $\bm{F}$.
  \end{itemize}
\end{theorem}

The proofs of all statements in Theorem \ref{thm:properties}
essentially rely on the following pointwise entropy inequality:

\begin{lemma}\label{lm:entreps-n-rot}
  Consider a (Lipschitz) solution of the non-local problem
  \eqref{eq:rotfilter2}, see Proposition \ref{prop:L1exist}.  Let
  $\eta\in C^2(\R)$ be a convex entropy function, and define the
  entropy fluxes $\frak{q}^\pm$ by
  $\partial_u\frak{q}^\pm(\by,u)
  =\eta'(u)\partial_u\frak{f}^\pm(\by,u)$ and
  $\frak{q}^{\pm}(\by,0)=0$.  We have that\footnote{Beware: The right
    hand side of \eqref{eq:entropy1eps-n-rot} does \emph{not} equal
    $\aud^\Phi\!\bm{Q}(u)$ where $\bm{Q}'=\eta'\bm{F}'$.  Indeed, our
    filter $\audp \bm{F}$ depends non-linearly on $\bm{F}$.}, for
  almost all $(\bx,t)$,
  \begin{equation}
    \label{eq:entropy1eps-n-rot}
    \begin{split}
      \partial_t \eta(u(\bx,t)) \le \int_{\R^d}
      &\frac{\Phi'(\abs{\by})}{\abs{\by}^d}
      \Bigl\{\frak{q}^+(\by,u(\bx,t))-\frak{q}^+(\by,u(\bx+\by,t))\\
      &\qquad\qquad +
      \frak{q}^-(\by,u(\bx,t))-\frak{q}^-(\by,u(\bx-\by,t))\Bigr\}
      \,d\bm{w}(\by).
    \end{split}
  \end{equation}
\end{lemma}

\begin{proof}
  A straightforward calculation shows that
  \begin{align*}
    \bigl(\frak{f}^\pm
    (\by,u(\bx,t))-\frak{f}^\pm
    &(\by,u(\bx\pm\by,t))\bigr)\eta'(u(\bx,t))
    \\ &=\frak{q}^\pm(\by,u(\bx,t))
         -\frak{q}^\pm(\by,u(\bx\pm\by,t))
    \\ & \qquad
         +\underbrace{\int_{u(\bx\pm\by,t)}^{u(\bx,t)}
         \int_s^{u(\bx,t)}\eta''(\sigma)\partial_u
         \frak{f}^\pm(\by,s)dsd\sigma}_{D^{\pm}}.
  \end{align*}
  Since the maps $u \mapsto\frak{f}^\pm(\cdot,u)$ are increasing, and
  since $\eta''\ge 0$, $D^\pm\ge 0$. Multiply \eqref{eq:rotfilter2}
  with $\eta'(u)$ to find 
  \begin{align*}
    \partial_t \eta(u(\bx,t))
    &= \int_{\R^d}
      \frac{\Phi'(\abs{\by})}{\abs{\by}^d}
      \Bigl\{\frak{q}^+(\by,u(\bx,t))-\frak{q}^+(\by,u(\bx+\by,t))
    \\ &\qquad\qquad 
         +\frak{q}^-(\by,u(\bx,t))
         -\frak{q}^-(\by,u(\bx-\by,t))\Bigr\}
         \,d\bm{w}(\by)\\ 
    &\qquad + \int_{\R^d} 
      \frac{\Phi'(\abs{\by})}{\abs{\by}^d}
      \left(D^++D^-\right)\,d\bm{w}(\by). 
  \end{align*}
  Since $\Phi'\le 0$ the last integral is non-positive.
\end{proof}

\begin{remark}
In the one-dimensional case 
\eqref{eq:suggestive}, we have a more explicit entropy 
balance equation (generalizing 
\cite[Theorem 5.1]{Coclite:2021aa}):
\begin{equation*}
	\pt \eta(u)+D_\alpha^++D_\alpha^-
	=\px A^+_\alpha{q_+(u)}+A^-_\alpha{q_-(u)},
\end{equation*}
where the operators $A^\pm_\alpha$ are defined 
in \eqref{eq:1D-average} and \eqref{eq:A-}. Additionally, 
the non-local dissipation terms $D_\alpha^\pm$ 
can be expressed as follows:
\begin{align*}
	D_+(x,t) & =
	\int_0^\infty \bigl(-\Phi_\alpha'\bigr)(\zeta)
	\int_{u(x,t)}^{u(x+y,t)}\int_{u(x,t)}^{\sigma}
	\eta''(\mu) f_+'(\sigma)\,d\mu\,d\sigma \,dy \ge 0,
	\\ 
	D_-(x,t)& = -
	\int_{-\infty}^0 \bigl(\Phi_\alpha'\bigr)(\zeta)
	\int_{u(x,t)}^{u(x+y,t)}\int_{u(x,t)}^{\sigma}
	\eta''(\mu) f_-'(\sigma)\,d\mu\,d\sigma \,dy \ge 0.
\end{align*}
These terms can be simplified if $\eta$ and $f$ 
are strictly convex functions, as discussed 
in \cite[Remark 5.2]{Coclite:2021aa} 
for monotone fluxes.
\end{remark}

Now we can prove part {\bf{b}} of Theorem~\ref{thm:properties}:
\begin{proof}[Proof {of \eqref{eq:L1bnd}}]
  We use the same notation as in the proof of Proposition
  \ref{prop:L1exist} to deduce
  \begin{align*}
    \pt\abs{\Delta(\bx,t)}
    &=\sgn{\Delta(\bx,t)}\Delta^\Phi(\bx,t)
    \\ & 
         =\int_{\R^d} \frac{\Phi'(\abs{\by})}{\abs{\by}^d}
         \Bigl\{\sgn{\Delta(\bx,t)}\left(\frak{f}^+(\by,u(\bx,t))
         -\frak{f}^+(\by,v(\bx,t))\right)
    \\ &\qquad\quad
         -\sgn{\Delta(\bx)}\left(\frak{f}^+(\by,u(\bx+\by,t))
         -\frak{f}^+(\by,v(\bx+\by,t))\right)\\
    &\qquad\quad
      +\sgn{\Delta(\bx)}\left(\frak{f}^-(\by,u(\bx,t))
      -\frak{f}^-(\by,v(\bx,t))\right)\\
    &\qquad\quad -\sgn{\Delta(\bx,t)}
      \left(\frak{f}^-(\by,u(\bx-\by,t))
      -\frak{f}^-(\by,v(\bx-\by,t))\right)
      \Bigr\}\,d\bm{w}(\by).
  \end{align*}
  Since $\frak{f}^\pm$ are both non-decreasing in $u$, we have that
  \begin{align*}
    \sgn{\Delta(\bx,t)}\left(\frak{f}^\pm(\by,u(\bx,t))
    -\frak{f}^\pm(\by,v(\bx,t))\right)
    =\abs{\frak{f}^\pm(\by,u(\bx,t))
    -\frak{f}^\pm(\by,v(\bx,t))},
  \end{align*}
  and trivially (since $\Phi'\le 0$)
  \begin{equation*}
    -\Phi'(r)\sgn{a}b\le -\Phi'(r)\abs{b} 
    \ \ \text{for all numbers $a$, $b$
      and non-negative $r$.}
  \end{equation*}
  Using this, we obtain
  \begin{align}
    \pt\abs{\Delta(\bx,t)}
    &\le \int_{\R^d} 
      \frac{\Phi'(\abs{\by})}{\abs{\by}^d}
      \Bigl\{\abs{\frak{f}^+(\by,u(\bx,t))
      -\frak{f}^+(\by,v(\bx,t))}\notag
    \\ &\qquad\qquad
         -\abs{\frak{f}^+(\by,u(\bx+\by,t))
         -\frak{f}^+(\by,v(\bx+\by,t))}\notag\\
    &\qquad\qquad
      +\abs{\frak{f}^-(\by,u(\bx,t))
      -\frak{f}^-(\by,v(\bx,t))}\label{eq:preLipschitz}\\
    &\qquad\qquad -\abs{\frak{f}^-(\by,u(\bx-\by,t))
      -\frak{f}^-(\by,v(\bx-\by,t))}\Bigr\}
      \,d\bm{w}(\by).\notag
  \end{align}

  Let $\delta>0$. Multiply \eqref{eq:preLipschitz} by
  $\psi(\bx)=e^{-\delta\abs{\bx}}$ and integrate in $\bx$ to obtain
  \begin{align*}
    \frac{d}{dt}
    &\iint_{\R^d}\abs{\Delta(\bx,t)}\psi(\bx)\,d\bx\\
    &\le \iint_{\R^{2d}}
      \frac{\Phi'(\abs{\by})}{\abs{\by}^d}
      \Bigl\{\abs{\frak{f}^+(\by,u(\bx,t))
      -\frak{f}^+(\by,v(\bx,t))}\psi(\bx)
    \\ &\qquad\qquad 
         -\abs{\frak{f}^+(\by,u(\bx+\by,t))
         -\frak{f}^+(\by,v(\bx+\by,t))}\psi(\bx)\\
    &\qquad\qquad
      +\abs{\frak{f}^-(\by,u(\bx,t))
      -\frak{f}^-(\by,v(\bx,t))}\psi(\bx)\\
    &\qquad\qquad -\abs{\frak{f}^-(\by,u(\bx-\by,t))
      -\frak{f}^-(\by,v(\bx-\by,t))}\psi(\bx)\Bigr\}
      \,d\bm{w}(\by)d\bx\\
    &= \iint_{\R^{2d}}\frac{\Phi'(\abs{\by})}{\abs{\by}^d}
      \Bigl\{\abs{\frak{f}^+(\by,u(\bx,t))-
      \frak{f}^+(\by,v(\bx,t))}
      \left(\psi(\bx)-\psi(\bx-\by)\right)
    \\ &\qquad\qquad
         +\abs{\frak{f}^-(\by,u(\bx,t))
         -\frak{f}^-(\by,v(\bx,t))}
         \left(\psi(\bx)-\psi(\bx+\by)\right)\Bigr\}
         \,d\bm{w}(\by)d\bx.
  \end{align*}
  By the convexity of the exponential function,
  \begin{align*}
    \abs{\psi(\bx\pm\by)-\psi(\bx)}
    &=\abs{e^{-\delta\abs{\bx\pm\by}}
      -e^{-\delta\abs{\bx}}}\\
    &\le \delta e^{-\delta\abs{\bx}}
      \absb{\abs{\bx\pm\by}-\abs{\by}}\\
    &\le \delta\psi(\bx)\abs{\by}.
  \end{align*}
  Recall the definition of 
  $\frak{f}^\pm$~\eqref{eq:frakfdef}, which implies
  \begin{align*}
    \abs{\frak{f}^\pm(\by,u)-\frak{f}^\pm(\by,v)}
    &=\abs{[\bm{F}\cdot\by]^\pm(u)
      -[\bm{F}\cdot\by]^\pm(v)}\\
    & \le \abs{\by} L \abs{u-v},
  \end{align*}
  where $L$ is the Lipschitz constant of $\bm{F}$. Using this
  \begin{align*}
    & \frac{d}{dt}
      \int_{\R^d}\abs{\Delta(\bx,t)}
      \psi(\bx)\,d\bx
    \\ & \quad 
         \le \delta\iint_{\R^{2d}}
         \frac{\abs{\Phi'(\abs{\by})}}{\abs{\by}^{d-1}}
         \Bigl\{\abs{\frak{f}^+(\by,u(\bx,t))
         -\frak{f}^+(\by,v(\bx,t))}
    \\ & \quad\qquad\qquad\qquad
         +\abs{\frak{f}^-(\by,u(\bx,t))
         -\frak{f}^-(\by,v(\bx,t))}
         \Bigr\}\psi(\bx)
         \,d\bm{w}(\by)d\bx 
    \\ & \quad
         \le 2L\delta\iint_{\R^{2d}}
         \frac{\abs{\Phi'(\abs{\by})}}
         {\abs{\by}^{d-2}}
         \abs{\Delta(\bx,t)} \psi(\bx)
         \,d\bm{w}(\by)d\bx
    \\ & \quad
         \le 2w(S^{d-1})\delta L
         \int_{\R^d}\abs{\Delta(\bx,t)}
         \psi(\bx)\,d\bx,
  \end{align*}
  since an easy calculation shows that
  \begin{equation*}
    \int_{\R^d} \frac{\abs{\Phi'(\abs{\by})}}
    {\abs{\by}^{d-2}} \,d\bm{w}(\by) = w(S^{d-1}).
  \end{equation*}
  Therefore, we can invoke Gronwall’s inequality to obtain
  \begin{equation*}
    \int_{\R^d}
    e^{-\delta|\bx|} \abs{u(\bx,t)-v(\bx,t)}\,d\bx
    \le e^{2w(S^{d-1})\delta Lt}
    \int_{\R^d}e^{-\delta|\bx|}
    \abs{u_0(\bx)-v_0(\bx)}\,d\bx. 
  \end{equation*}
  This estimate holds for $\delta>0$ and any two solutions $u$ and
  $v$. If $u_0-v_0\in L^1(\R^d)$, we can send $\delta\to0$ to conclude
  the proof.
\end{proof}

Now, we can establish \eqref{eq:BVbnd} and \eqref{eq:TLipbnd}.

\begin{proof}[Proof of \eqref{eq:BVbnd} and \eqref{eq:TLipbnd}]
  Setting $v(\bx,t)=u(\bx+\bm{h},t)$ in \eqref{eq:L1bnd} shows
  \eqref{eq:BVbnd}.

  To show \eqref{eq:TLipbnd} we integrate \eqref{eq:rotfilter2} with
  respect to $\bx$ and then apply the triangle inequality:
  \begin{align*}
    \int_{\R^d}
    &\abs{\pt u(\bx,t)}\,d\bx
    \\ &\le  \iint_{\R^{2d}}
         \frac{\abs{\Phi'(\abs{\by})}}{\abs{\by}^d}
         \Bigl\{\abs{\frak{f}^+(\by,u(\bx,t))
         -\frak{f}^+(\by,u(\bx+\by,t))}
    \\ & \hphantom{\int_{\R^d}
         \frac{\Phi'(\abs{\by})}{\abs{\by}^d}
         \Bigl\{} \qquad
         + \abs{\frak{f}^-(\by,u(\bx,t))
         -\frak{f}^-(\by,u(\bx-\by,t))}\Bigr\}
         \,d\bm{w}(\by) d\bx
    \\ &\le 
         L \iint_{\R^{2d}}
         \frac{\abs{\Phi'(\abs{\by})}}{\abs{\by}^{d-1}}
         \Bigl\{\abs{u(\bx,t)-u(\bx+\by,t)}
    \\ & \hphantom{\int_{\R^d}
         \frac{\Phi'(\abs{\by})}{\abs{\by}^d}
         \Bigl\{}
         \qquad
         +\abs{u(\bx,t) -u(\bx-\by,t)}\Bigr\}
         \,d\bx d\bm{w}(\by)
    \\ &
         \le 2L\abs{u_0}_{BV}\int_{\R^d}
         \frac{ \abs{\Phi'(\abs{\by})}}
         {\abs{\by}^{d-2}}\,d\bm{w}(\by)
    \\ &=2w(S^{d-1})L\abs{u_0}_{BV},
  \end{align*}
  which concludes the proof.
\end{proof}

The following entropy estimate is useful in proving \eqref{eq:supbnd}

\begin{lemma}\label{lm:entreps-1-n-rot}
  Let $\eta\in C^2(\R)$ be a convex non-negative entropy function. If
  $u_0$ and $\eta(u_0)$ both are in $L^1(\R^d)$, then
  \begin{equation}\label{eq:entropy2eps-n-rot}
    \int_{\R^d}\eta(u(\bx,t))d\bx
    \le \int_{\R^d}\eta(u_0(\bx))d\bx,
    \ \ \text{for $t\ge 0$}.
  \end{equation}
\end{lemma}

\begin{proof}
  Integrate the pointwise entropy inequality
  \eqref{eq:entropy1eps-n-rot} over $\bx$, then change the order of
  integration and finally change variables $\bx\mapsto \bx\pm\by$ to
  obtain
  \begin{align*}
    \frac{d}{dt}\int_{\R^d}
    &\eta(u(\bx,t))\,d\bx\\
    &\le \int_{\R^d}\int_{\R^d}
      \frac{\Phi'(\abs{\by})}{\abs{\by}^d}
      \Bigl\{\frak{q}^+(\by,u(\bx,t))
      -\frak{q}^+(\by,u(\bx+\by,t))
    \\ &\qquad\qquad 
         +\frak{q}^-(\by,u(\bx,t))
         -\frak{q}^-(\by,u(\bx-\by,t))\Bigr\}
         \,d\bm{w}(\by)d\bx
    \\ &=\int_0^\infty\!\Phi(r)\int_{S^{d-1}}
         \!\Bigl[\int_{\R^d}
         \Bigl\{\frak{q}^+(\bn_\theta,u(\bx,t))
         -\frak{q}^+(\bn_\theta,u(\bx+r\bn_\theta,t))
    \\ &\qquad\qquad 
         + \frak{q}^-(\bn_\theta,u(\bx,t))
         -\frak{q}^-(\bn_\theta,u(\bx-r\bn_\theta,t))\Bigr\}\,d\bx 
         \Bigr]\,dw(\theta)dr=0,
  \end{align*}
  where we have used that $\frak{q}^\pm(\by,u(\cdot,t)) \in L^1(\R^d)$
  for all $t$. This holds since $\frak{q}^\pm$ is Lipschitz, and
  $u(\cdot,t)$ is in $L^1(\R^d)$ by \eqref{eq:L1bnd}.
\end{proof}

\begin{proof}[Proof of \eqref{eq:supbnd}]
  Set $\hat{u}=\essup_{\bx} u_0(\bx)$ and
  $\check{u}=\essinf_{\bx} u_0(\bx)$.  Choose a convex entropy $\eta$
  such that
  \begin{equation*}
    \eta(u)=\begin{cases}
      0 &\text{if $u \in [\check{u},\hat{u}]$},
      \\ >0 &\text{otherwise}.
    \end{cases}
  \end{equation*}
  Since $\int \eta(u_0)\,d\bx=0$, \eqref{eq:entropy2eps-n-rot} implies
  that $u(\bx,t)\in [\check{u},\hat{u}]$ for a.e.~$(\bx,t)$.
\end{proof}

The proof of part {\bf{c}}, which concludes the proof of
Theorem~\ref{thm:properties}, requires 
us to use a weak formulation of \eqref{eq:rotfilter2}.
To this end, we multiply \eqref{eq:rotfilter2} with a test
function $\test\in C^\infty_0(\R^d\times [0,\infty))$ 
and integrate by parts in $t$ to get:
\begin{align*}
    \int_0^\infty \!\!\!\int_{\R^d}& u(\bx,t)\pt\vfi\,d\bx dt
    +\int_{\R^d} {u}_{0}(\bx)\vfi(\bx,0)\,d\bx\\
    &=-\int_0^\infty \!\!\!\int_{\R^d} \audp
    \bm{F}(u)(\bx,t)\test(\bx,t)\,d\bx dt.
\end{align*}
Through transformations of variables, 
such as $\bx\mapsto \bx\pm\by$,
this can be rewritten as
\begin{align*}
  \int_0^\infty
  &\!\!\!\int_{\R^d}
    u(\bx,t)\pt\vfi\,d\bx dt
    +\int_{\R^d} {u}_{0}(\bx)\vfi(\bx,0)\,d\bx
  \\ & 
       =\int_0^\infty\!\!\!\iint_{\R^{2d}}
       \frac{\Phi'(\abs{\by})}{\abs{\by}^d}
       \Bigl\{\frak{f}^+(\by,u(\bx,t))\left(\test(\bx-\by,t)
       -\test(\bx,t)\right)
  \\ &\hphantom{=
       \int_0^\infty\!\!\!\iint_{\R^{2d}}
       \frac{\Phi_\alpha'(\abs{\by})}{\abs{\by}^d}
       \Big\{}
       +\frak{f}^-(\by,u(\bx,t))\left(\test(\bx+\by,t)
       -\test(\bx,t)\right) \Bigr\}
       \,d\bx d\bm{w}(\by)dt.
\end{align*}
We can rewrite the right hand side as follows:
\begin{align*}
  \int_{\R^d}
  &\frac{\Phi'(\abs{\by})}{\abs{\by}^d}
    \frak{f}^+(\by,u(\bx,t))\left(\test(\bx-\by,t)
    -\test(\bx,t)\right)
    \,d\bm{w}(\by)
  \\ &=-\int_{S^{d-1}}\!\!\int_0^\infty
       \Phi'(r) \frak{f}^+\left(\bn_\theta,u(\bx,t)\right)
       \int_0^r \bn_\theta\cdot \nabla 
       \test\left(\bx-s\bn_\theta\right)\,ds
       \,drdw(\theta) 
  \\& =\int_{S^{d-1}}
  \frak{f}^+\left(\bn_\theta,u(\bx,t)\right)
  \bn_\theta\cdot 
  \left(\int_0^\infty \Phi(r) \nabla 
  \test\left(\bx+r\bn_\theta\right)
  \,dr\right) \,dw(\theta),
\end{align*}
where we integrated by parts. 
Since we have a similar formula for the term with $\frak{f}^-$, 
which is
\begin{align*}
  \int_{\R^d}
  &\frac{\Phi'(\abs{\by})}{\abs{\by}^d}
    \frak{f}^-(\by,u(\bx,t))\left(\test(\bx+\by,t)
    -\test(\bx,t)\right)
    \,d\bm{w}(\by)
  \\ & =-\int_{S^{d-1}}\!\!
       \frak{f}^-\left(\bn_\theta,u(\bx,t)\right)
       \bn_\theta\cdot \left(\int_0^\infty 
       \Phi(r) \nabla \test\left(\bx+r\bn_\theta\right)
       \,dr\right) \,dw(\theta),
\end{align*}
we obtain a weak formulation of \eqref{eq:rotfilter2}, 
in which all the derivatives have been 
moved onto the test function 
$\test\in C^\infty_c(\R^d\times [0,\infty))$:
\begin{align}
  \int_0^\infty
  &\!\!\!\int_{\R^d} u(\bx,t)\pt\vfi\,d\bx dt
    +\int_{\R^d} {u}_{0}(\bx)\vfi(\bx,0)\,d\bx\notag
  \\ &=\int_0^\infty \!\!\!\int_{\R^d}\int_{S^{d-1}}
       \Biggl(\frak{f}^+\left(\bn_\theta,u(\bx,t)\right) 
       \bn_\theta\cdot \left(\int_0^\infty \Phi(r)
       \nabla \test\left(\bx-r\bn_\theta\right)\,dr\right)\notag
  \\ &\hphantom{=-\int_0^\infty \!\!\!\int_{\R^d} }
       -\frak{f}^-\left(\bn_\theta,u(\bx,t)\right)
       \bn_\theta\cdot \left(\int_0^\infty \Phi(r)
       \nabla \test\left(\bx+r\bn_\theta\right)\,dr\right)
       \Biggr)\,dw(\theta)\,d\bx dt\notag 
  \\ &=\int_0^\infty \!\!\!\int_{\R^d}\int_{\R^d}
       \frac{\Phi(\abs{\by})}{\abs{\by}^{d+1}}\; \by\cdot\Bigl[
       \frak{f}^+\left(\by,u(\bx,t)\right)
       \nabla\test\left(\bx-\by\right)\label{eq:avetest}
  \\ &\hphantom{=\int_0^\infty 
       \!\!\!\int_{\R^d}\int_{\R^d}
       \frac{\Phi(\abs{\by})}{\abs{\by}^{d+1}}\; \by\cdot\Bigl[ }
       -\frak{f}^-\left(\by,u(\bx,t)\right) \nabla
       \test\left(\bx+\by\right) \Bigr]
       \,d \bm{w}(\by)d\bx dt.\notag
\end{align}

We shall also use a weak formulation of the entropy inequality
\eqref{eq:entropy1eps-n-rot}.  If we multiply
\eqref{eq:entropy1eps-n-rot} with a non-negative test function
$\test\in C^\infty(\R^d\times [0,\infty))$, integrate over $\bx$ and
$t$, integrate by parts in $t$, and proceed as we did to deduce
\eqref{eq:avetest}, we obtain the following $(\bx,t)$-weak form of the
pointwise entropy inequality \eqref{eq:entropy1eps-n-rot}:
\begin{align}
  \int_0^\infty
  &\!\!\!\int_{\R^d} \eta(u(\bx,t))\pt\vfi\,d\bx dt
    +\int_{\R^d} \eta({u}_{0}(\bx))\vfi(\bx,0)\,d\bx\notag
  \\ &\ge\int_0^\infty \!\!\!\int_{\R^d}\int_{\R^d}
       \frac{\Phi(\abs{\by})}{\abs{\by}^{d+1}}\; \by\cdot\Bigl[
       \frak{q}^+\left(\by,u(\bx,t)\right)
       \nabla\test\left(\bx-\by\right)
       \notag
  \\ &\hphantom{=\int_0^\infty \!\!\!\int_{\R^d}\int_{\R^d}
       \frac{\Phi(\abs{\by})}{\abs{\by}^{d+1}}\; \by\cdot\Bigl[ }
       -\frak{q}^-\left(\by,u(\bx,t)\right)
       \nabla\test\left(\bx+\by\right) 
       \Bigr]\,d\bm{w}(\by)d\bx dt\notag\\
  &= \int_0^\infty \!\!\!\int_{\R^d}\int_{S^{d-1}}
       \Biggl(\frak{q}^+\left(\bn_\theta,u(\bx,t)\right) 
       \bn_\theta\cdot \left(\int_0^\infty \Phi(r)
       \nabla \test\left(\bx-r\bn_\theta\right)
       \,dr\right)\label{eq:weakentropy}
  \\ &\hphantom{=-\int_0^\infty \!\!\!\int_{\R^d} }
       -\frak{q}^-\left(\bn_\theta,u(\bx,t)\right)
       \bn_\theta\cdot \left(\int_0^\infty \Phi(r)
       \nabla \test\left(\bx+r\bn_\theta\right)\,dr\right)
       \Biggr)\,dw(\theta)\,d\bx dt.\notag 
\end{align}

\begin{proof}[Proof of \eqref{eq:Phistab}]
  Given the stability estimate \eqref{eq:L1bnd}, we can simplify the
  analysis by considering the case where the initial functions are
  identical, i.e., $u_0=v_0\in BV(\R^d)$. Our proof draws inspiration
  from Kuznetsov’s lemma \cite{Kuznetsov:1976ys}.

  Let $\omega$ be a standard mollifier on $\R$ and define
  \begin{equation*}
    \test(\bx,t,\bz,s)=\omega_{\eps_0}(t-s)
    \bm{\omega}_\eps(\bx-\bz),
    \qquad 
    \bm{\omega}_\eps(\bx)=\prod_{i=1}^d 
    \omega_\eps(x_i).
  \end{equation*}
  For any final time $T>0$ and suitable functions $u=u(\bx,t)$ and
  $\phi=\phi(\bx,t)$, we introduce the functional
  \begin{align*}
    \Lambda(\Phi,\phi,u,k)
    &=\int_{\R^d}\abs{u(\bx,T)-k}\phi(\bx,T)
      \,d\bx-\int_{\R^d}\abs{u(\bx,0)-k}\phi(\bx,0) \,d\bx 
    \\ &\quad - \int_0^T\int_{\R^d}\abs{u-k}
         \partial_t \varphi\,d\bx dt
         +\Gamma(\Phi,\phi,u,k),
  \end{align*}
  where
  \begin{align*}
    \Gamma(\Phi,\phi,u,k)
    &= \int_0^T \!\!\!\int_{\R^d}\int_{\R^d}
      \frac{\Phi(\abs{\by})}{\abs{\by}^{d+1}}\; \by\cdot\Bigl[
      \frak{q}^+\left(\by,u(\bx,t),k\right)
      \nabla\phi\left(\bx-\by\right)
    \\ &\hphantom{=\int_0^\infty \!\!\!
         \int_{\R^d}\int_{\R^d}
         \frac{\Phi(\abs{\by})}{\abs{\by}^{d+1}}\; 
         \by\cdot\Bigl[ }
         -\frak{q}^-\left(\by,u(\bx,t),k\right) 
         \nabla \phi\left(\bx+\by\right) 
         \Bigr]\,d\bm{w}(\by)d\bx dt.
  \end{align*}
  Here, $\frak{q}^\pm(\by,u,k)$ denotes the entropy flux associated
  with the entropy $u\mapsto \abs{u-k}$:
  \begin{equation*}
    \frak{q}^\pm(\by,u,k)=\sgn{u-k}
    \left((\by\cdot \bm{F})^\pm(u)
      -(\by\cdot \bm{F})^\pm(k)\right).
  \end{equation*}

  Let $v=v(\bx,t)$ be the solution of the non-local equation
  \eqref{eq:rotfilter2} with filter $\Psi$ (and initial data $u_0$),
  and note that in this case
  \begin{equation*}
    \Lambda(\Psi,
    \test(\cdot,\bz,\cdot,s),v,k)\le 0,
  \end{equation*}
  for all $(\bz,s)\in \R^d\times (0,\infty)$ and $k\in \R$.  Set
  \begin{equation*}
    \Lambda_{\eps_0,\eps}(\Psi,u,v)
    =\int_0^T\int_{\R^d}\Lambda(\Psi,
    \test(\cdot,\bz,\cdot,s),u,v(\bz,s))\,d\bz ds,
  \end{equation*}
  and
  \begin{equation*}
    \Gamma_{\eps_0,\eps}(\Psi,v,u)=
    \int_0^T\int_{\R^d}\Gamma(\Psi,
    \test(\cdot,\bz,\cdot,s),u,v(\bz,s))\,d\bz ds.
  \end{equation*}
  Choose $u=u(\bz,s)$ to be an entropy solution of
  \eqref{eq:rotfilter2} with filter $\Psi$ (and initial data $u_0$). 
  This results in
  \begin{align}
    0&\ge \Lambda_{\eps_0,\eps}(\Psi,v,u)
       +\Lambda_{\eps_0,\eps}(\Phi,u,v)\notag
    \\ & 
         =\Lambda_{\eps_0,\eps}(\Psi,v,u)
         -\Lambda_{\eps_0,\eps}(\Phi,v,u)\notag
    \\ &\quad
         + \int_0^T\int_{\R^{2d}}\abs{v(\bx,T)-u(\bz,s)}
         \omega_{\eps_0}(T-s)\bm{\omega}_\eps(\bx-\bz)
         \, d\bx d\bz ds \label{eq:t1}
    \\ &\quad
         +\int_0^T\int_{\R^{2d}}\abs{u(\bx,T)-v(\bz,s)}
         \omega_{\eps_0}(T-s)\bm{\omega}_\eps(\bx-\bz)
         \, d\bx d\bz ds \label{eq:t2}
    \\ &\quad 
         -\int_0^T\int_{\R^{2d}}\abs{u_0(\bx)-u(\bz,s)}
         \omega_{\eps_0}(s)\bm{\omega}_\eps(\bx-\bz)
         \, d\bx d\bz ds \label{eq:t3}
    \\ &\quad
         -\int_0^T\int_{\R^{2d}}\abs{u_0(\bx)-v(\bz,s)}
         \omega_{\eps_0}(s)\bm{\omega}_\eps(\bx-\bz)
         \,d\bx d\bz ds. \label{eq:t4}
  \end{align}
  Since $u$ and $v$ are of bounded variation in space and $L^1$
  time-continuous, we obtain the bounds
  \begin{align*}
    \abs{\eqref{eq:t1}
    +\eqref{eq:t2}}
    &\ge 
      \norm{u(\cdot,T)-v(\cdot,T)}_{L^1(\R^d)}
      -C(\eps_0+\eps),
    \\ \abs{\eqref{eq:t3}+\eqref{eq:t4}}
    &\ge -C(\eps_0+\eps),
  \end{align*}
  where the constant $C$ does not 
  depend on $\Phi$ or $\Psi$. We also have
  \begin{align*}
    &\Lambda_{\eps_0,\eps}
    (\Psi,v,u)-\Lambda_{\eps_0,\eps}(\Phi,v,u)
    \\ &= \Gamma_{\eps_0,\eps}(\Psi-\Phi,v,u)
    \\ &=\int_0^T\int_0^T \int_{\R^{2d}}\int_{\R^d}
         \frac{\Psi(\abs{\by})-\Phi(\abs{\by})}{\abs{\by}^{d+1}}\,\by
         \cdot\Bigl[\frak{q}^+\left(\by,u(\bx,t),v(\bz,s)\right)
         \nabla\bm{\omega}_\eps(\bx-\bz-\by) 
    \\&\hphantom{=\int_0^T\int_0^T \int_{\R^{2d}}\int_{\R^d}}
    -\frak{q}^-\left(\by,u(\bx,t),v(\bz,s)\right) 
    \nabla\bm{\omega}_\eps\left(\bx-\bz-\by\right)\Bigr] 
    \\ &\hphantom{=\int_0^T\int_0^T \int_{\R^{2d}}\int_{\R^d}}
        \qquad \times\bm{\omega}_{\eps_0}(t-s)\,d\bm{w}(\by) d\bx d\bz dsdt
    \\ &=\int_0^T\int_0^T \int_{\R^{2d}}\int_{0}^\infty\int_{S^{d-1}}
         \left(\Psi(r)-\Phi(r)\right)\bn_\theta
    \\ & \hphantom{=\int_0^T\int_0^T 
         \int_{\R^{2d}}\int_0^\infty\int}
         \cdot \Bigl[
         \frak{q}^+\left(\bn_\theta,u(\bx,t),v(\bz,s)\right)
         \nabla\bm{\omega}_\eps(\bx-\bz-r\bn_\theta) 
    \\ & \hphantom{=\int_0^T\int_0^T 
         \int_{\R^{2d}}\int_0^\infty\int_{S^(d-1)}}
         -\frak{q}^-\left(\bn_\theta,u(\bx,t),v(\bz,s)\right)
         \nabla\bm{\omega}_\eps\left(\bx-\bz+r\bn_\theta\right)
         \Bigr]\\ &\hphantom{=\int_0^T\int_0^T 
                    \int_{\R^{2d}}\int_0^\infty\int_{S^(d-1)}}
                    \qquad \times\bm{\omega}_{\eps_0}(t-s)
                    \,dw(\theta)dr d\bx d\bz dsdt. 
  \end{align*}
  Next we have that \newcommand{\ints}{\int_0^T\!\!\!\int_0^T\!\!\!
    \int_{\R^{2d}}\!\int_{0}^\infty \!\!\!\int_{S^{d-1}}\!\!\!\!}
  \begin{align*}
    0&=\int_0^\infty \left(\Psi(r)-\Phi(r)\right)\,dr
    \\ &=\ints\left(\Psi(r)-\Phi(r)\right)\bn_\theta
         \cdot \Bigl[
         \frak{q}^+\left(\bn_\theta,u(\bx,t),v(\bz,s)\right)
         \nabla\bm{\omega}_\eps(\bx-\bz)
    \\ &\hphantom{=\int_0^T\int_0^T
         \int_{\R^{2d}}\int_0^\infty\int_{S^(d-1)}}
         -\frak{q}^-\left(\bn_\theta,u(\bx,t),v(\bz,s)\right) \nabla
         \bm{\omega}_\eps\left(\bx-\bz\right)
         \Bigr]
    \\ &
         \hphantom{=\int_0^T\int_0^T \int_{\R^{2d}}
         \int_0^\infty\int_{S^(d-1)}}
         \qquad\times\omega_{\eps_0}(t-s)
         \,dw(\theta)dr d\bx d\bz dsdt. 
  \end{align*}
  Consequently, we deduce that
  \begin{align*}
    \bigl|\Gamma_{\eps_0,\eps}
    &(\Psi-\Phi,v,u)\bigr|
    \\ &= \Bigl|\ints \left(\Psi(r)-\Phi(r)\right)
         \Bigl[\frak{q}^+\left(\bn_\theta,u(\bx,t),v(\bz,s)\right)
    \\ &\hphantom{=\ints}
      \quad  \times \bn_\theta\cdot
      \nabla\left(\bm{\omega}_\eps(\bx-\bz-r\bn_\theta)
         -\bm{\omega}_\eps(\bx-\bz)\right)
    \\ &\hphantom{=\ints}
    -\frak{q}^-\left(\bn_\theta,u(\bx,t),v(\bz,s)\right)
    \\ &\hphantom{=\ints}
         \quad \times \bn_\theta\cdot\nabla
         \left(\bm{\omega}_\eps(\bx-\bz+r\bn_\theta)
         -\bm{\omega}_\eps(\bx-\bz)\right)\Bigr]
    \\ &\hphantom{=\ints}
         \quad\quad \times\omega_{\eps_0}(t-s)
         \,dw(\theta)dr d\bx d\bz dsdt\Bigr|
    \\ & =\Bigl|\ints \left(\Psi(r)-\Phi(r)\right)
         \bn_\theta\cdot\nabla\bm{\omega}_\eps(\bx-\bz)
    \\ & 
         \qquad \qquad 
         \times \Bigl\{
         \bigl[
         \frak{q}^+\left(\bn_\theta,u(\bx+r\bn_\theta,t),v(\bz,s)\right)
         -\frak{q}^+\left(\bn_\theta,u(\bx,t),v(\bz,s)\right)\bigr]
    \\ &\qquad\qquad\quad
         +\bigl[\frak{q}^-\left(\bn_\theta,u(\bx-r\bn_\theta,t),v(\bz,s)\right)
         -\frak{q}^-\left(\bn_\theta,u(\bx,t),v(\bz,s)\right)
         \bigr]\Bigr\}
    \\ 
    &\hphantom{=\ints}
      \ \ \times\omega_{\eps_0}(t-s)
      \,dw(\theta)dr d\bx d\bz dsdt\Bigr|
    \\ &\le \ints \abs{\Phi(r)-\Psi(r)}\,L
         \abs{\nabla\bm{\omega}_\eps(\bx-\bz)}
    \\ & \qquad\qquad\quad
         \times \Bigl\{\bigl[
         \abs{u(\bx+r\bn_\theta,t)-u(\bx,t)}
         +\abs{u(\bx-r\bn_\theta,t)-u(\bx,t)}\bigr]\Bigr\}
    \\ &\hphantom{=\ints}\ \ 
         \times\omega_{\eps_0}(t-s)
         \,dw(\theta)dr d\bx d\bz dsdt.
  \end{align*}
  To estimate this we first integrate $\omega_{\eps_0}$ in $s$ and
  then $\abs{\nabla\bm{\omega}_\eps(\bx-\bz)}$ in $\bz$ using
  \begin{equation*}
    \int_{\R^d}\abs{\nabla 
      \bm{\omega}_{\eps}(\bx-\bz)}\,d\bz 
    = \frac{d^{\frac32}\,2\omega(0)}{\eps}.
  \end{equation*}
  To prove this estimate first note that
  \begin{align*}
    \abs{\nabla 
    \bm{\omega}_{\eps}(\bx)}
    &= \Bigl(\sum_{i=1}^d \left(\partial_{x_i}\bm{\omega}
    (\bx)\right)^2\Bigr)^{\frac12}
    \le \Bigl(d\max_i\left\{
      \abs{\partial_{x_i}\bm{\omega}(\bx)}^2\right\}\Bigr)^{\frac12}\\
    &\le \sqrt{d}\sum_{i=1}^d \abs{\partial_{x_i}\bm{\omega}(\bx)}.
  \end{align*}
  Next compute
  \begin{align*}
    \int_{\R_d}\abs{\nabla 
    \bm{\omega}_{\eps}(\bx-\bz)}\,d\bz
    &\le \sqrt{d} \sum_{i=1}^d\int_{\R^d}
      \abs{\partial_{x_i}\bm{\omega}_\eps(\bx-\bz)}\,d\bz\\
    &= \sqrt{d} \sum_{i=1}^d \prod_{j\ne i}\Bigl(\int_\R
      \omega_\eps(x_j-z)\,dz\Bigr)
      \int_\R \abs{\omega_\eps'\left(x_i-z\right)}\,dz\\
    &= d^{\frac32}\frac{2\omega(0)}{\eps}.
  \end{align*}
  Then we integrate the differences in $u$ with respect to $\bx$ using
  that $u(\cdot,t)\in BV(\R^d)$ with
  $\abs{u(\cdot,t)}_{BV}\le \abs{u_0}_{BV}$,
  \begin{equation*}
    \int_{\R^d}\abs{u(\bx\mp r\bn_\theta,t)
      -u(\bx,t)}\,d\bx\le r\abs{u_0}_{BV(\R^d)}.
  \end{equation*}
  After this we end up with the estimate
  \begin{equation*}
    \bigl|\Gamma_{\eps_0,\eps}
    (\Phi-\Psi,v,u)\bigr|\le
    \frac{4LTd^{\frac32}\abs{u_0}_{BV(\R^d)}\omega(0)}{\eps}
    \int_0^\infty r\abs{\Phi(r)-\Psi(r)}\,dr.
  \end{equation*}
  We can then set $\eps_0=0$ and rearrange \eqref{eq:t1} --
  \eqref{eq:t4} to get
  \begin{equation*}
    \norm{u(\cdot,T)-v(\cdot,T)}_{L^1(\R^d)}\le C\Bigl(
    \eps + \frac{1}{\eps}\int_0^\infty 
    r\abs{\Phi(r)-\Psi(r)}\,dr\Bigr).
  \end{equation*}
  Minimizing over $\eps$ concludes the proof.
\end{proof}

\section{The zero filter limit}\label{sec:filtlim}

We now examine the effect of scaling the filter $\Phi$, under the
assumptions that $\Phi$ satisfies the conditions \eqref{eq:Phi-r},
\eqref{eq:Phi-m} and \eqref{eq:Phi-zero}.  To this end, we introduce
the rescaled filter
\begin{equation}\label{eq:scalePhi}
  \Phi_\alpha(r)=\frac{1}{\alpha}
  \Phi\Bigl(\frac{r}{\alpha}\Bigr), 
  \quad \alpha\in (0,1],
\end{equation}
Denote by $\ua$ the solution of the corresponding upwind filtered
conservation law
\begin{equation*}
  \pt\ua =\auda \bm{F}(\ua),
  \qquad\ua(\cdot,0)=u_0,
\end{equation*}
where $\auda:=\aud^{\Phi_\alpha}$ is defined by
\eqref{eq:non-local-div}. We assume that $u_0\in BV(\R^d)$.

Due to the non-negativity of $\Phi$ and the fact that the integral of
$\Phi$ is one, cf.~\eqref{eq:Phi-r}, $\Phi_\alpha$ tends to the Dirac
measure $\delta_0$ located at zero as $\alpha\to 0$, in the sense of
distributions on $[0,\infty)$. 
Hence, by \eqref{eq:avetest} and \eqref{eq:weakentropy}, and assuming
$u=\lim_{\alpha\to 0}\ua$, we formally have that
\begin{align*}
  &\int_0^\infty
    \!\!\!\int_{\R^d} \eta(\ua(\bx,t))\pt\vfi\,dtd\bx
    +\int_{\R^d} \eta({u}_{0}(\bx))\vfi(\bx,0)\,d\bx 
  \\ & \quad \ge\int_0^\infty \!\!\!\int_{\R^d}\int_{S^{d-1}}\Biggl(
       \frak{q}^+\left(\bn_\theta,\ua(\bx,t)\right) \bn_\theta
       \cdot\left(\int_0^\infty \Phi_\alpha(r) \nabla
       \test\left(\bx-r\bn_\theta,t\right)\,dr\right)\\
  &\hphantom{=- \!\!\!\int_{\R^d}}\quad
    -\frak{q}^-\left(\bn_\theta,\ua(\bx,t)\right) \bn_\theta
    \cdot \left(\int_0^\infty \Phi_\alpha(r)
    \nabla\test\left(\bx+r\bn_\theta,t\right)\,dr\right)
    \Biggr)\,dw(\theta)\,d\bx dt
  \\ &\hphantom{=- \!\!\!\int_{\R^d}
       -\frak{q}^-\left(\bn_\theta,\ua(\bx,t)
       \right) \bn_\theta\cdot}\quad
       \bigg\downarrow \qquad\ \text{($\alpha\to 0$)} 
  \\ & \int_0^\infty 
       \!\!\!\int_{\R^d} \eta(u(\bx,t))\pt\vfi\,dtd\bx
       +\int_{\R^d} \eta({u}_{0}(\bx))\vfi(\bx,0)\,d\bx 
  \\ & \, \, 
       \ge\int_0^\infty \!\!\!\int_{\R^d}\int_{S^{d-1}}
       \Bigl( \frak{q}^+\left(\bn_\theta,u(\bx,t)\right)
       -\frak{q}^-\left(\bn_\theta,u(\bx,t)\right) \Bigr)
       \left( \bn_\theta\cdot \nabla \test(\bx,t)\right) 
       \,dw(\theta)\,d\bx dt
  \\ & \,\, 
       =\int_0^\infty \!\!\!\int_{\R^d}\int_{S^{d-1}}
       \Bigl(\bn_\theta\cdot \bm{Q}(u(\bx,t)) \Bigr)
       \left( \bn_\theta\cdot
       \nabla\test(\bx,t) \right)
       \,dw(\theta)\,d\bx dt
  \\ & \,\, =\int_0^\infty \!\!\!\int_{\R^d}
       \Bigl(\sum_{i,j=1}^d
       q_i(u(\bx,t))\left\{\int_{S^{d-1}}
       \left(\bn_\theta\cdot \bm{e}_i\right)
       \left(\bn_\theta\cdot \bm{e}_j\right)
       \,dw(\theta)\right\}\,
       \test_{x_j}(\bx,t)\Bigr)\,d\bx dt
  \\ & \,\, =\int_0^\infty \!\!\!\int_{\R^d}
       \bm{Q}(u(\bx,t))\cdot\nabla\test(\bx,t)
       \,d\bx dt,
\end{align*}
where we have used the symmetry of the measure $w$, concretely
equation \eqref{eq:scaling}. We have also 
introduced the standard \textit{entropy flux}
$\bm{Q}=(q_1,\ldots,q_d):\R\to \R^d$ (for local conservation laws)
that corresponds to an entropy function $\eta:\R\to \R$, defined as
follows:
$$
q_i'=\eta' f_i', \quad i=1,\ldots,d.
$$ 
Thus, if $\ua$ converges strongly to $u$ as $\alpha\to 0$, then $u$ is
an entropy solution to the local conservation law.  For the reader’s
convenience, we define an entropy solution.

\begin{definition}
  A function $u\in C([0,T];L^1(\R^d)) \cap L^\infty(\R^d\times \R^+)$,
  for all $ T>0$, is a solution of the local conservation law
  \begin{equation}\label{eq:conslaw}
    \pt u = \aud\bm{F}(u),
    \qquad 
    u(\bx,0)=u_0(\bx),
  \end{equation}
  if, for any convex entropy/entropy flux pair $(\eta,\bm{Q})$, the
  inequality
  \begin{equation*}
    \pt \eta(u)\le \dv{\bm{Q}(u)}
  \end{equation*}
  holds in the distributional sense in $\R^d\times [0,\infty)$.
\end{definition}

The following theorem guarantees the strong convergence of $\ua$ as
the filter size $\alpha$ tends to zero.

\begin{theorem}\label{th:mainCL}
Suppose the filter $\Phi$ is a non-negative and non-increasing
function on $\R^+$ that satisfies \eqref{eq:Phi-r},
\eqref{eq:Phi-zero}.  For every
$u_0\in (L^1\cap L^\infty\cap BV)(\R^d)$, we have that
\begin{equation}\label{eq:convmain}
	\text{$u_{\alpha}\to u$ 
	in $C([0,T];L^1(\R^d))$,\,\,\, for all $T>0$},         
\end{equation}
where $u_{\alpha}$ is the solution of the non-local conservation law
\eqref{eq:rotfilter2} with initial data $u_{0}$ and rescaled filter
\eqref{eq:scalePhi}, and $u$ is the unique entropy solution of
\eqref{eq:conslaw} with the same initial data $u_{0}$.  Moreover,
the following error estimate holds:
\begin{equation}\label{eq:error}
	\norm{\ua(\cdot,t)-u(\cdot,t)}_{L^1(\R^d)}
	\le C\abs{u_0}_{BV(\R^d)}\sqrt{\alpha t},
\end{equation}
for a constant $C$ that does not depend on $u$.
\end{theorem}

\begin{proof}
  We start by proving \eqref{eq:convmain}.  We specify
  $\Phi=\Phi_\alpha$ and $\Psi=\Phi_\beta$ in the continuous
  dependence estimate \eqref{eq:Phistab} to obtain
  \begin{align}
    \norm{\ua(\cdot,t)-u_\beta(\cdot,t)}^2_{L^1(\R^d)}
    &\le Ct\int_0^\infty 
      r\abs{\Phi_\alpha(r)-\Phi_\beta(r)}
      \,dr\notag 
    \\ &\le Ct\int_0^\infty 
         r\left(\Phi_\alpha(r)+\Phi_\beta(r)\right)\,dr
         \label{eq:u1}
    \\ &\le 2Ct\int_0^\infty r
         \Phi_{\max\seq{\alpha,\beta}}(r)\,dr 
         \ \longrightarrow 0,\notag
  \end{align}
  as $\max\seq{\alpha,\beta}\to 0$.

  Thus, for any sequence $\seq{\ak}$ with $\ak\to 0$ as $k\to\infty$,
  $\seq{\uak}_{\ak}$ is a Cauchy sequence in
  $C([0,T];L^1(\R^d))$.  Note that since $\uak\to u$ strongly,
  \begin{equation*}
    \frak{q}^\pm(\by,\uak)\to \frak{q}^\pm(\by,u)\ \ 
    \text{in $C([0,T];L^1(\R^d))$}.
  \end{equation*}
  Then we can proceed as indicated above to conclude that the limit
  $u$ is the entropy solution to the local conservation law
  \eqref{eq:conslaw}.

  To prove the convergence rate \eqref{eq:error}, observe that
  \begin{equation*}
    \int_0^\infty r\Phi_\beta(r)\,dr
    =\beta \int_0^\infty r\Phi(r)\,dr.
  \end{equation*}
  Consequently, \eqref{eq:u1} implies that
  \begin{equation*}
    \norm{\ua(\cdot,t)-u_\beta(\cdot,t)}^2_{L^1(\R^d)}
    \le C t\left(\alpha+\beta\right)
    \int_0^\infty r\Phi(r)\,dr.
  \end{equation*}
  Sending $\beta\to 0$ gives the desired estimate.
\end{proof}

\section{Examples with links to numerical methods}
\label{sec:examp}

Upwind filtered conservation laws \eqref{eq:rotfilter2} can serve as
non-local models in various contexts (see the discussion in the
introduction). However, the goal of this section is to demonstrate how
certain well-known numerical methods emerge as specific instances of
\eqref{eq:rotfilter2}.

Recall that the Engquist-Osher flux \cite{EngquistOsher:81}
$f^\EO:\R^2\to \R$ takes the form
\begin{equation*}
  f^\EO(a,b)=f^+(a)+f^-(b),
\end{equation*}
where the monotone functions $f^\pm$ are defined in
\eqref{eq:EO-split}.  Let $\Phi_\alpha$ be the rescaled filter from
\eqref{eq:scalePhi}. In general, $\auda\bm{F}(u)$ can be rewritten as
follows:
\begin{align*}
  \auda \bm{F}(u)(\bx)
  & =\int_{\R^d}
    \frac{\Phi_\alpha'(\abs{\by})}{\abs{\by}^{d-1}}
    \Bigl\{[\bn_\theta\cdot\bm{F}]^-(u(\bx-\by))
    +[\bn_\theta\cdot\bm{F}]^+(u(\bx))
  \\ &\hphantom{\int_{\R^d}
       \frac{\Phi_\alpha'(\abs{\by})}{\abs{\by}^d}
       \Bigl\{}\quad
       -[\bn_\theta\cdot\bm{F}]^-(u(\bx))
       +[\bn_\theta\cdot\bm{F}]^+(u(\bx+\by))
       \Bigr\}\,d\bm{w}(\by)
  \\ &=- \int_{\R^d}
       \frac{\Phi_\alpha'(\abs{\by})}{\abs{\by}^{d-1}}
       \Bigl\{[\bn_\theta\cdot\bm{F}]^\EO\!\!\left(u(\bx),
       u(\bx+\by)\right)
  \\ &\hphantom{\int_{\R^d}
       \frac{\Phi_\alpha'(\abs{\by})}{\abs{\by}^d}
       \Bigl\{}\qquad
       -[\bn_\theta\cdot\bm{F}]^\EO\!\!
       \left(u(\bx-\by),u(\bx)\right)\Bigr\}\,d\bm{w}(\by)
  \\ & =-\int_0^\infty \Phi'_\alpha(r)\int_{S^{d-1}}
       [\bn_\theta\cdot\bm{F}]^\EO
       \!\!\left(u(\bx),u(\bx+\bn_\theta r)\right)
  \\ &\hphantom{-\int_0^\infty 
       \Phi_\alpha(r)\int_{S^{d-1}}}\quad
       -[\bn_\theta\cdot\bm{F}]^\EO\!\!
       \left(u(\bx-\bn_\theta r),u(\bx)\right)\,dw(\theta)dr,
\end{align*}
for an integrable function $u=u(\bx)$.  This formulation highlights
the connection between the non-local divergence operator $\auda$ and
numerical (difference) schemes for conservation laws.  Below we will
illustrate how specific choices of the filter $\Phi$ and the measure
$w$ can lead to practical schemes.

Note that the above formula ``integrates everything twice'', and can
be simplified using
\begin{equation*}
  [-\bn_\theta\cdot \bm{F}]^\EO(a,b)
  =-[\bn_\theta\cdot \bm{F}]^\EO(b,a),
\end{equation*}
to
\begin{align*}
  \auda \bm{F}(u)(\bx) 
  & =-\int_0^\infty \Phi'_\alpha(r)\int_{S^{d-1}_+}
  [\bn_\theta\cdot\bm{F}]^\EO\!\!
  \left(u(\bx),u(\bx+\bn_\theta r)\right)
  \\ &\hphantom{-\int_0^\infty 
  \Phi_\alpha(r)\int_{S^{d-1}}}\quad
  -[\bn_\theta\cdot
  \bm{F}]^\EO\!\!\left(u(\bx-\bn_\theta r),u(\bx)\right)
  \,d\omega(\theta)dr,
\end{align*}
where the ``omega" measure $\omega$ is defined as
$\omega(U)=w(U)+w(-U)$ for any set $U\subset S^{d-1}_+$, and
$S^{d-1}_+$ denotes the half-sphere
\begin{equation*}
  S^{d-1}_+=\seq{\bx\in S^{d-1}\, \bigm|\, x_d\ge 0}.
\end{equation*}

\subsection{One space dimension}\label{subsec:1d}

If $d=1$, then $\bm{F}(u)=f_1(u) =:f(u)$, $\bm{n}=\seq{1,-1}$, and
\begin{equation*}
  \frak{f}^{\pm}(\by,u)
  =\pm\abs{y}\left[\bm{n}f\right]^\pm(u).
\end{equation*}
Furthermore, $S^0=\seq{1,-1}$ and ${w}$ is the counting measure
divided by two, so that $w(S^0)=1$.  Thus the filtered upwind
divergence of $f(u)(x)$ reads
\begin{equation*}
  \begin{aligned}
    \auda{f(u)}(x) &=-\int_\R\Phi_\alpha'(\abs{y})
    \bigl\{\left[\bm{n}f\right]^+(u(x))
    -\left[\bm{n}f\right]^+(u(x+y)) \\ &\hphantom{=-\sum_{\bm{n}\in
        S^0} \int_\R\Phi_\alpha'(\abs{y}) \bigl\{}
    -\left[\bm{n}f\right]^-(u(x))
    +\left[\bm{n}f\right]^-(u(x-y))\bigr\}\,d\bm{w}(y).
  \end{aligned}
\end{equation*}

Choosing $\Phi$ to be the characteristic function of the interval
$[0,1]$, we obtain
\begin{equation*}
  -\int_0^\infty \Phi_\alpha'(r) h(r)\,dr 
  = \frac{h(\alpha)}{\alpha}
\end{equation*}
if $h$ is continuous. We know that $\ua$ is continuous for $\alpha>0$,
so that in this case \eqref{eq:rotfilter2} becomes
\begin{equation*}
  \pt \ua(x,t)
  =\frac{1}{\alpha}\left(f^\EO(\ua(x,t),\ua(x+\alpha,t))
    -f^\EO(\ua(x-\alpha,t),\ua(x,t))\right),
\end{equation*}
which is a semi-discrete form of the Engquist-Osher scheme
\cite{EngquistOsher:81}.

Other choices of $\Phi$ can yield other schemes.  If we choose
$\Phi(r)=\max\seq{2(1-r),0}$, then
\begin{equation*}
  -\int_0^\infty \Phi_\alpha'(r) h(r)\,dr =
  \frac{2}{\alpha^2}\int_0^\alpha h(r)\,dr =
  \frac{2}{\alpha}\left(\frac1\alpha\int_0^\alpha
    h(r)\,dr\right)=:
  \frac{2}{\alpha} \overline{h}.
\end{equation*}
This gives the scheme
\begin{align*}
  \pt \ua(x,t)
  = \frac{f^+(\ua(x,t))
  -\ob{f^+(\ua(x-\cdot,t))}}{\alpha/2}
  +\frac{\ob{f^-(\ua(x+\cdot,t))}
  -f^-(\ua(x,t))}{\alpha/2},
\end{align*}
where, in practice, the averaging integrals would typically be
approximated using a quadrature rule.

\subsection{Two space dimensions}\label{subsec:2d}

These examples can be extended to the case $d>1$. For $d=2$, let $w$
denote the point measure concentrated at the points
\begin{equation*}
  \bn_1=(1,0),\, \, \bn_2=(0,1),\,\,
  \bn_3=-\bn_1,\,\, 
  \text{and}\,\, \bn_4=-\bn_2,
\end{equation*}
scaled such that \eqref{eq:scaling} holds, which implies that
$w(S^1)=2$.  In this scenario, the upwind divergence
\eqref{eq:non-local-div} simplifies to
\begin{align*}
  \auda \bm{F}(u)(\bx)
  &=-\frac12 \int_0^\infty \Phi_\alpha'(r)
    \Bigl\{ f_1^\EO(u(\bx),u(\bx+r\bn_1))
    -f_1^\EO(u(\bx-r\bn_1),u(\bx))
  \\ &\hphantom{\int_0^\infty
       \Phi_\alpha'(r)}\quad
       +f_2^\EO(u(\bx),u(\bx+r\bn_2))
       -f_2^\EO(u(\bx-r\bn_2),u(\bx))\Bigr\}\,dr,
\end{align*}
for an integrable function $u=u(x)$.  If we select $\Phi$ as the
characteristic function of $[0,1]$, we obtain the conventional
Engquist-Osher scheme in two dimensions on a square grid.

This approach works also works for equilateral hexagons. Let
\begin{equation*}
  \bn_i=\bigl(\cos(\pi(i-1)/3),
  \sin(\pi(i-1)/3)\bigr),
  \quad i=1,\ldots,6.
\end{equation*}
In this case, we set $w = \frac{1}{3} \sum_{i=1}^6 \delta_{\bn_i}$ to
ensure that $w(S^1) = 2$. Consequently, the average upwind divergence
of $\bm{F}(u)(\bx)$ transforms into
\begin{align*}
  \auda \bm{F}(u)(\bx) 
  &=-\frac13 \int_0^\infty \Phi_\alpha'(r)
    \Biggl\{\sum_{i=1}^3
    \left(\bn_i\cdot\bm{F}\right)^\EO(u(\bx),u(\bx+r\bn_i))
  \\ &\hphantom{=\frac13 \int_0^\infty 
       \Phi_\alpha'(r)\sum_{i=1}^3\Bigl\{}\quad 
       -\left(\bn_i\cdot\bm{F}\right)^\EO
       (u(\bx-r\bn_i),u(\bx))\Biggr\}\,dr.
\end{align*}
It is (perhaps) surprising that if we choose the measure $w$ to have
support on the corners of an equilateral triangle, and $\Phi$ to be
the characteristic function of the unit interval, we recover the
scheme for hexagons.  In Figure~\ref{fig:grids}, we have shown the
stencils for the filter $\Phi= \chi_{[0,1]}$. We note that the
resulting schemes take the form of finite volume schemes rather than
finite difference schemes.  The extension of these constructions to
$d>2$ is obvious, but complicated.

\begin{figure}[h!] \label{fig:grids} \centering
  \begin{tabular}[h!]{lcr}
    \includegraphics[width=0.30\linewidth]{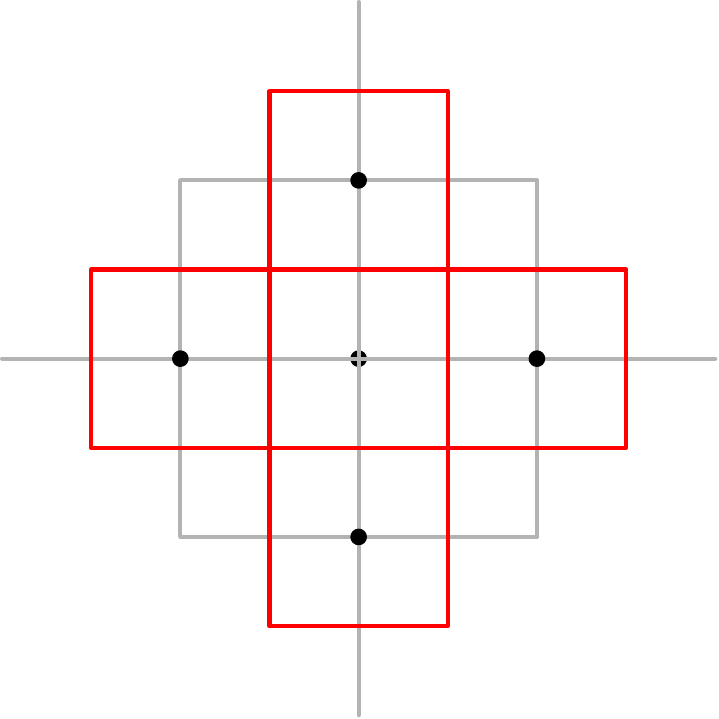}
    &\includegraphics[width=0.30\linewidth]{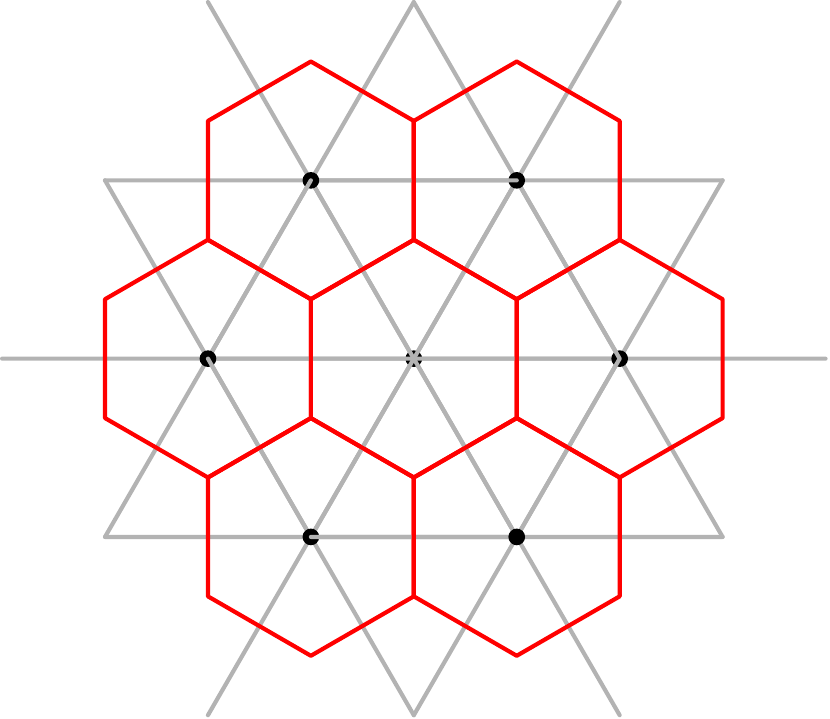}
    &\includegraphics[width=0.30\linewidth]{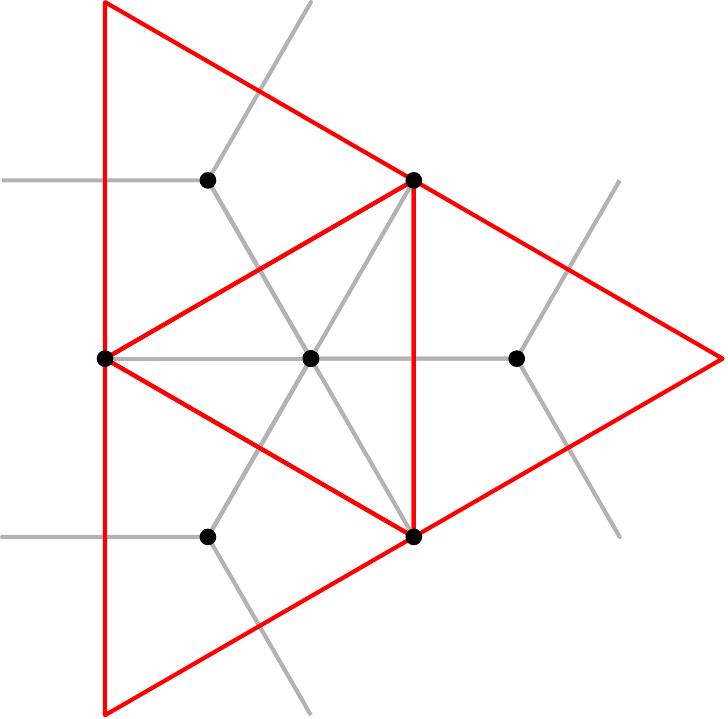}
  \end{tabular}
  \caption{Stencils for the schemes obtained when $w$ is a point
    measure concentrated at points on the unit circle and
    $\Phi= \chi_{[0,1]}$.  Left: points on a square, middle: points on
    a hexagon, and right: points on a triangle. The finite volumes are
    marked in red and the grid in gray.}
\end{figure}

\section{Unfiltered variables}\label{sec:bonus}

In this final section, we will discuss in more detail the 
one-dimensional case of \eqref{eq:suggestive}, where 
the filtered conservation law \eqref{eq:non-local-CL} 
assumes the following form:
\begin{equation}\label{eq:1D-bonus}
	\partial_t u= \partial_x A^+_\alpha(f^+(u))
	+\partial_xA^-_\alpha (f^-(u)).
\end{equation}
We will focus on the kernel 
$\Phi_\alpha(\zeta)=e^{-\zeta/\alpha}/\alpha$,
which is similar to those used in 
\cite{Bressan:2020aa,Bressan:2021aa,Coclite:2023ab}. 
With this kernel, the non-local 
operators \eqref{eq:1D-average} and \eqref{eq:A-} 
transform into the following forms:
\begin{align*}
	& A_\alpha^+v(x) := \int_0^\infty
	\frac{e^{-\zeta/\alpha}}{\alpha}v(x+\zeta)\,d\zeta
	=\int_x^\infty \frac{e^{(x-\zeta)/\alpha}}{\alpha}
	v(\zeta) \, d\zeta,
	\\ & 
	A_\alpha^-{v}(x) 
	= \int_{-\infty}^0 \frac{e^{\zeta/\alpha}}{\alpha}
	v(x+\zeta)\,d\zeta
	=\int^x_{-\infty} \frac{e^{-(x-\zeta)/\alpha}}{\alpha}
	v(\zeta)\, d\zeta.
\end{align*}
Introducing the first order differential operators
\begin{equation*}
	\mathbb{A}^+_\alpha
	=-\alpha\px+I,\qquad 
	\mathbb{A}^-_\alpha=\alpha\px+I,
\end{equation*}
we can easily verify that
$$
(\mathbb{A}_\alpha^+)^{-1}v=A_\alpha^+v, 
\qquad 
(\mathbb{A}_\alpha^-)^{-1}v=A_\alpha^-v.
$$
Simple computations reveal that 
\begin{equation}\label{eq:bonus-relations}
	\begin{split}
		&\mathbb{A}^+_\alpha
		\circ (\mathbb{A}^-_\alpha)^{-1} v
		=\mathbb{A}^+_\alpha\circ A^-_\alpha v
		=-v+2 A^-_\alpha v,
		\\
		&\mathbb{A}^-_\alpha \circ 
		(\mathbb{A}^+_\alpha)^{-1}v
		=\mathbb{A}^-_\alpha\circ A^+_\alpha v
		=-v+2 A^+_\alpha v.
	\end{split}
\end{equation}
Applying $\mathbb{A}_\alpha^+$ and $\mathbb{A}_\alpha^-$ 
separately to \eqref{eq:1D-bonus} and utilizing the 
relations \eqref{eq:bonus-relations}, we obtain the 
following two equations:
\begin{align*}
	& \pt \mathbb{A}_\alpha^+(u)
	=\px \bigl(f_+(u)-f^-(u)\bigr)
	+2\px A^-_\alpha(f^-(u)),
	\\
	& \pt(A_\alpha^-u)=
	\px  \bigl( f^-(u)-f_+(u)\bigr)
	+2\px A^+_\alpha(f_+(u)),
\end{align*}
noting that
$$
f^+(u) - f^-(u) = \int_0^{u} 
\abs{f'(v)} \, dv:=\widetilde{f}(u).
$$
Finally, we introduce notations for the “inverse-filtered” 
unknowns $\mathbb{A}^\pm_\alpha (u)$:
\begin{equation*}
	U^+:=\mathbb{A}^+_\alpha (u),
	\qquad 
	U^-=\mathbb{A}^-_\alpha(u),
\end{equation*}
which leads to following system of equations:
\begin{equation}\label{eq:bonus-non-filter}
	\begin{split}
		&\pt U^+
		=\px \left( \widetilde{f}(u)
		+2A^-_\alpha (f^-(u))\right),
		\\ & 
		\pt U^-
		=-\px \left(\widetilde{f}(u) 
		+2A^+_\alpha(f_+(u))\right).
	\end{split}
\end{equation}
Consider now the case where $f$ is a non-decreasing function, so 
that $f^- = 0$ and $\tilde{f} = f$. Then the 
first equation in \eqref{eq:bonus-non-filter} simplifies to
$$
\partial_t U^+ = \partial_x f(u),
$$
where $U^+$ should be regarded as the unfiltered 
variable, and $u$ should be considered as the filtered variable 
($u = A^+_\alpha(U^+)$). This corresponds to the traffic flow 
equations \eqref{eq:y-nofilter} and \eqref{eq:1D-filter-prev} 
derived in \cite{Coclite:2023aa}. 
Similarly, if $f$ is non-decreasing, applying the 
second equation in \eqref{eq:bonus-non-filter} gives us:
$$
\partial_t U^-_\alpha = \partial_x f(u).
$$
This approach of ``inverting" the non-local operator 
appears to be effective only in the 
case of an exponential kernel.

\section*{Acknowledgments}

{\small GMC is member of the Gruppo Nazionale per l'Analisi Matematica, la
Probabilit\`a e le loro Applicazioni (GNAMPA) of the Istituto
Nazionale di Alta Matematica (INdAM). GMC has been partially supported
by the Project funded under the National Recovery and Resilience Plan
(NRRP), Mission 4 Component 2 Investment 1.4 -Call for tender No. 3138
of 16/12/2021 of Italian Ministry of University and Research funded by
the European Union -NextGenerationEUoAward Number: CN000023,
Concession Decree No. 1033 of 17/06/2022 adopted by the Italian
Ministry of University and Research, CUP: D93C22000410001, Centro
Nazionale per la Mobilit\`a Sostenibile, the Italian Ministry of
Education, University and Research under the Programme Department of
Excellence Legge 232/2016 (Grant No. CUP - D93C23000100001), and the
Research Project of National Relevance ``Evolution problems involving
interacting scales'' granted by the Italian Ministry of Education,
University and Research (MIUR Prin 2022, project code 2022M9BKBC,
Grant No. CUP D53D23005880006). GMC expresses his 
gratitude to the Department of Mathematics at the 
University of Oslo and the Hamburg Institute for 
Advanced Study (HIAS) for their warm hospitality. This 
research was funded by the Research Council of Norway
under project 351123 (NASTRAN).}


\bibliographystyle{abbrv}

\begin{thebibliography}{10}

\bibitem{Betancourt:2010kx}
F.~Betancourt, R.~B{\"u}rger, K.~H. Karlsen, and E.~M. Tory.
\newblock On nonlocal conservation laws modeling sedimentation.
\newblock {\em Nonlinearity}, 24(3):855--885, 2011.

\bibitem{BlandinGoatin:16}
S.~Blandin and P.~Goatin.
\newblock Well-posedness of a conservation law with non-local flux arising in
  traffic flow modeling.
\newblock {\em Numer. Math.}, 132(2):217--241, 2016.

\bibitem{Bouchut:1998ys}
F.~Bouchut and B.~Perthame.
\newblock Kru\v zkov's estimates for scalar conservation laws revisited.
\newblock {\em Trans. Amer. Math. Soc.}, 350(7):2847--2870, 1998.

\bibitem{Bressan:2020aa}
A.~Bressan and W.~Shen.
\newblock On traffic flow with nonlocal flux: a relaxation representation.
\newblock {\em Arch. Ration. Mech. Anal.}, 237(3):1213--1236, 2020.

\bibitem{Bressan:2021aa}
A.~Bressan and W.~Shen.
\newblock Entropy admissibility of the limit solution for a nonlocal model of
  traffic flow.
\newblock {\em Commun. Math. Sci.}, 19(5):1447--1450, 2021.

\bibitem{Chiarello:0aa}
F.~A. Chiarello.
\newblock An overview of non-local traffic flow models.
\newblock In {\em Mathematical descriptions of traffic flow: micro, macro and
  kinetic models}, volume~12 of {\em ICIAM 2019 SEMA SIMAI Springer Ser.},
  pages 79--91. Springer, Cham, 2021.

\bibitem{ChiarelloGoatin:18}
F.~A. Chiarello and P.~Goatin.
\newblock Global entropy weak solutions for general non-local traffic flow
  models with anisotropic kernel.
\newblock {\em ESAIM Math. Model. Numer. Anal.}, 52(1):163--180, 2018.

\bibitem{Coclite:2023ab}
G.~M. Coclite, J.-M. Coron, N.~De~Nitti, A.~Keimer, and L.~Pflug.
\newblock A general result on the approximation of local conservation laws by
  nonlocal conservation laws: the singular limit problem for exponential
  kernels.
\newblock {\em Ann. Inst. H. Poincar\'{e} C Anal. Non Lin\'{e}aire},
  40(5):1205--1223, 2023.

\bibitem{Coclite:2021aa}
G.~M. Coclite, N.~De~Nitti, A.~Keimer, and L.~Pflug.
\newblock Singular limits with vanishing viscosity for nonlocal conservation
  laws.
\newblock {\em Nonlinear Anal.}, 211:Paper No. 112370, 12, 2021.

\bibitem{Coclite:2017aa}
G.~M. Coclite and E.~Jannelli.
\newblock Well-posedness for a slow erosion model.
\newblock {\em J. Math. Anal. Appl.}, 456(1):337--355, 2017.

\bibitem{Coclite:2023aa}
G.~M. Coclite, K.~H. Karlsen, and N.~H. Risebro.
\newblock A nonlocal {L}agrangian traffic flow model and the zero-filter limit.
\newblock {\em Z. Angew. Math. Phys.}, 75(66), 2024.

\bibitem{Colombo:2021aa}
M.~Colombo, G.~Crippa, E.~Marconi, and L.~V. Spinolo.
\newblock Local limit of nonlocal traffic models: convergence results and total
  variation blow-up.
\newblock {\em Ann. Inst. H. Poincar\'{e} C Anal. Non Lin\'{e}aire},
  38(5):1653--1666, 2021.

\bibitem{Colombo:2019aa}
M.~Colombo, G.~Crippa, and L.~V. Spinolo.
\newblock On the singular local limit for conservation laws with nonlocal
  fluxes.
\newblock {\em Arch. Ration. Mech. Anal.}, 233(3):1131--1167, 2019.

\bibitem{Colombo:2012aa}
R.~M. Colombo, M.~Garavello, and M.~L\'ecureux-Mercier.
\newblock A class of nonlocal models for pedestrian traffic.
\newblock {\em Math. Models Methods Appl. Sci.}, 22(4):1150023, 34, 2012.

\bibitem{Colombo:2011aa}
R.~M. Colombo, M.~Herty, and M.~Mercier.
\newblock Control of the continuity equation with a non local flow.
\newblock {\em ESAIM Control Optim. Calc. Var.}, 17(2):353--379, 2011.

\bibitem{Crippa:2013aa}
G.~Crippa and M.~L\'ecureux-Mercier.
\newblock Existence and uniqueness of measure solutions for a system of
  continuity equations with non-local flow.
\newblock {\em NoDEA Nonlinear Differential Equations Appl.}, 20(3):523--537,
  2013.

\bibitem{Du:2017aa}
Q.~Du, Z.~Huang, and P.~G. LeFloch.
\newblock Nonlocal conservation laws. {A} new class of monotonicity-preserving
  models.
\newblock {\em SIAM J. Numer. Anal.}, 55(5):2465--2489, 2017.

\bibitem{EngquistOsher:81}
B.~Engquist and S.~Osher.
\newblock One-sided difference approximations for nonlinear conservation laws.
\newblock {\em Math. Comp.}, 36(154):321--351, 1981.

\bibitem{FriedrichEtal:18}
J.~Friedrich, O.~Kolb, and S.~G\"{o}ttlich.
\newblock A {G}odunov type scheme for a class of {LWR} traffic flow models with
  non-local flux.
\newblock {\em Netw. Heterog. Media}, 13(4):531--547, 2018.

\bibitem{GoatinScialanga:16}
P.~Goatin and S.~Scialanga.
\newblock Well-posedness and finite volume approximations of the {LWR} traffic
  flow model with non-local velocity.
\newblock {\em Netw. Heterog. Media}, 11(1):107--121, 2016.

\bibitem{Keimer:2017aa}
A.~Keimer and L.~Pflug.
\newblock Existence, uniqueness and regularity results on nonlocal balance
  laws.
\newblock {\em J. Differential Equations}, 263(7):4023--4069, 2017.

\bibitem{Keimer:2019aa}
A.~Keimer and L.~Pflug.
\newblock On approximation of local conservation laws by nonlocal conservation
  laws.
\newblock {\em J. Math. Anal. Appl.}, 475(2):1927--1955, 2019.

\bibitem{Keimer:2018aa}
A.~Keimer, L.~Pflug, and M.~Spinola.
\newblock Existence, uniqueness and regularity of multi-dimensional nonlocal
  balance laws with damping.
\newblock {\em J. Math. Anal. Appl.}, 466(1):18--55, 2018.

\bibitem{Keimer:2018ab}
A.~Keimer, L.~Pflug, and M.~Spinola.
\newblock Nonlocal scalar conservation laws on bounded domains and applications
  in traffic flow.
\newblock {\em SIAM J. Math. Anal.}, 50(6):6271--6306, 2018.

\bibitem{Kloeden:2016aa}
P.~E. Kloeden and T.~Lorenz.
\newblock Nonlocal multi-scale traffic flow models: analysis beyond vector
  spaces.
\newblock {\em Bull. Math. Sci.}, 6(3):453--514, 2016.

\bibitem{Kruzkov:1970kx}
S.~N. Kru{\v{z}}kov.
\newblock First order quasilinear equations with several independent variables.
\newblock {\em Mat. Sb. (N.S.)}, 81 (123):228--255, 1970.

\bibitem{Kuznetsov:1976ys}
N.~N. Kuznetsov.
\newblock The accuracy of certain approximate methods for the computation of
  weak solutions of a first order quasilinear equation.
\newblock {\em U.S.S.R. Computational Math. and Math. Phys.}, 16(6):105--119,
  1976.

\bibitem{Lighthill:1955aa}
M.~J. Lighthill and G.~B. Whitham.
\newblock On kinematic waves. {II}. {A} theory of traffic flow on long crowded
  roads.
\newblock {\em Proc. Roy. Soc. London Ser. A}, 229:317--345, 1955.

\bibitem{Lorenz:2020aa}
T.~Lorenz.
\newblock Viability in a non-local population model structured by size and
  spatial position.
\newblock {\em J. Math. Anal. Appl.}, 491(1):124249, 50, 2020.

\bibitem{Lucier:1986px}
B.~J. Lucier.
\newblock A moving mesh numerical method for hyperbolic conservation laws.
\newblock {\em Math. Comp.}, 46(173):59--69, 1986.

\bibitem{Pata:2019aa}
V.~Pata.
\newblock {\em Fixed point theorems and applications}, volume 116 of {\em
  Unitext}.
\newblock Springer, Cham, 2019.

\bibitem{Piccoli:2011aa}
B.~Piccoli and A.~Tosin.
\newblock Time-evolving measures and macroscopic modeling of pedestrian flow.
\newblock {\em Arch. Ration. Mech. Anal.}, 199(3):707--738, 2011.

\bibitem{Richards:1956aa}
P.~I. Richards.
\newblock Shock waves on the highway.
\newblock {\em Operations Research}, 4(1):42--51, February 1956.

\bibitem{Zumbrun:1999aa}
K.~Zumbrun.
\newblock On a nonlocal dispersive equation modeling particle suspensions.
\newblock {\em Quart. Appl. Math.}, 57(3):573--600, 1999.

\end{thebibliography}

 \end{document}